\documentclass[11pt,twoside, a4paper, english, reqno]{amsart}
\usepackage{a4wide}
\usepackage{amscd}
\usepackage{amssymb}
\usepackage{amsthm}
\usepackage{graphicx}
\usepackage{amsmath,calrsfs}
\usepackage{latexsym}
\usepackage{enumitem,dsfont}

\usepackage{upref}
\usepackage[colorlinks]{hyperref}
\usepackage{color,graphics}
\usepackage{upref,hyperref,color}
\usepackage[usenames,dvipsnames]{xcolor}
\usepackage{pdfsync}
\usepackage{ulem,cancel,pgf}
\usepackage{frcursive}

\theoremstyle{plain}
\newtheorem{theorem}{Theorem}[section]
\theoremstyle{plain}
\newtheorem{lemma}[theorem]{Lemma}
\newtheorem{prop}[theorem]{Proposition}
\newtheorem{cor}[theorem]{Corollary}

\theoremstyle{definition}
\newtheorem{definition}{Definition}[section]

\newtheorem{remark}{Remark}[section]

\newtheorem*{maintheorem*}{Main Theorem}
\newtheorem*{maincorollary*}{Main Corollary}

\numberwithin{equation}{section} \allowdisplaybreaks

\title[Optimal control  of  2D third grade fluids ]
{Optimal control  of  two dimensional   third grade fluids   }

\date{\today }

\thanks{
}

\author[Yassine Tahraoui]{ Yassine Tahraoui}
\address[ Yassine Tahraoui]{\newline
Center for Mathematics and Applications (NovaMath), FCT NOVA}
\email[Yassine Tahraoui]{tahraouiyacine@yahoo.fr}

\author[Fernanda Cipriano]{Fernanda Cipriano}
\address[Fernanda Cipriano]{\newline
Center for Mathematics and Applications (NovaMath), FCT NOVA and Department of Mathematics, FCT NOVA}
\email[Fernanda Cipriano]{cipriano@fct.unl.pt}

\usepackage{comment}
\begin{document}
\begin{abstract}
		The aim of this work is to study the optimal control problems of flows  governed by the incompressible third grade fluid  equations with Navier-slip boundary conditions.  After recalling a result on the  well-posedness of the state equations,  we study the existence and the uniqueness of solution to the 	linearized state  and  adjoint equations. Furthermore, we present a  stability result for the state,  and show that the solution of the linearized equation coincides with the G\^ateaux derivative of the control-to-state mapping. 
		Next, we prove the existence of an optimal solution and establish the first order optimality conditions. Finally,  an uniqueness result of the coupled system constituted by the state equation, the adjoint equation and  the first order optimality condition is established.
\end{abstract} 
\maketitle
\hspace*{0.45cm}\textbf{MSC:} 35Q35,49K20,76A05, 76D55 \\

\tableofcontents
\section{Introduction} 
In this work, we are concerned with the optimal control of the velocity field $y$   of a non-Newtonian fluid
filling a two-dimensional bounded domain  with a smooth boundary.
 More precisely, we consider a tracking  problem and the aim is  to minimize the following  cost functional
\begin{align*}
J(U,y)=\dfrac{1}{2}\int_0^T\Vert y-y_d\Vert_2^2dt+\dfrac{\lambda}{2} \int_0^T\Vert U\Vert_2^2dt,
\end{align*}
where  $y_d \in (L^2(D\times (0,T)))^2$ corresponds to  a desired target velocity field,  $\lambda \geq 0$
  sets the intensity of the cost, the control acts through the external force $U,$ and the velocity field $y$ is constrained to satisfy
the incompressible third grade fluid equation
\begin{align}
\label{111}
\begin{cases}
&\partial_t(v(y))-\nu \Delta y+(y\cdot \nabla)v(y)+\displaystyle\sum_{j=1}^2v(y)^j\nabla y^j-(\alpha_1+\alpha_2)\text{div}(A(y)^2) -\beta \text{div}[tr(A(y)A(y)^T)A(y)]\\[0.2cm]
&\qquad= -\nabla \mathbf{P} +U,\quad
v(y):=y-\alpha_1\Delta y, \quad A(y):= \nabla y+\nabla y^T,
\end{cases}
\end{align}
where the constant $\nu$ represents the fluid viscosity, $\alpha_1,\alpha_2$, $\beta$ are  the material moduli, and  $\mathbf{P}$ denotes the pressure. The equation will be supplemented with  a divergence free initial condition, and  a homogeneous  Navier-slip boundary condition
which allows the slippage of the fluid against the boundary wall (see Section \ref{S3} for more details). \\

 Most studies on fluid dynamics have been devoted to Newtonian fluids, which are characterized by the classical Newton's  law of viscosity,
establishing a linear relation between the shear stress and the strain rate. However,
there exist many real, industrial, or physiological fluids  with nonlinear viscoelastic behavior  that does not obey Newton's law of viscosity, and consequently  cannot be described by the classical  viscous Newtonian fluid model. These fluids include natural biological fluids such as blood, geological flows and others, and arise 
	in  polymer processing, coating, colloidal suspensions and emulsions, ink-jet prints, etc. (see e.g \cite{DR95,FR80,RKWA17}). Therefore, it is necessary to consider  more general fluid models. \\
	
	 Recently, the class of non-Newtonian fluids of differential type has received a special attention,  since it could be related to the viscous Camassa and Holm equation, shallow water models, geodesic motion on
	the volume-preserving diffeomorphism group for a metric containing the $H^1$-norm of the fluid velocity  (see \cite{Busuioc, CL17, Hol-Mar-Rat-98}) and it  found to be useful in turbulence theory, see \cite{CH1998}. 
	In order to describe the evolution of this special type of fluids,  we consider  the velocity field $y$ of	the fluid, and introduce the Rivlin-Ericksen kinematic tensors in \cite{RE55} $A_n, n\geq 1$, defined by  
			\begin{align*}
			\left\{\begin{array}{ll}
			A_1(y)&=\nabla y+\nabla y^T,
			\\A_n(y)&=\dfrac{d}{dt}A_{n-1}(y)+A_{n-1}(y)(\nabla y)+(\nabla y)^TA_{n-1}(y), \quad n=2,3,\cdots
			\end{array}
			\right.
			\end{align*}
		The constitutive law of fluids of grade $n$ 
		 reads
		\begin{align*}
		\mathbb{T}=-pI+F(A_1,\cdots,A_n), \quad 
			\end{align*}
		where  $F$ is an isotropic polynomial function    of degree $n$, subject to the usual	requirement of material frame indifference.
			The constitutive law of   third grade fluid is given by the following  equation	 
				\begin{align*}
				\mathbb{T}=-pI+\nu A_1+\alpha_1A_2+\alpha_2A_1^2+\beta_1 A_3+\beta_2(A_1A_2+A_2A_1)+\beta_3tr(A_1^2)A_1,
				\end{align*}
			where $\mathbb{T}$ is the shear stress
			tensor and	$(\alpha_i)_{1,2}, (\beta_i)_{1,2,3}$ are material moduli. 			The momentum equations are given by
				$$\dfrac{Dy}{Dt}=\dfrac{dy}{dt}+y\cdot \nabla y=div(\mathbb{T}).$$
				
				If $\beta_i=0,i=1,2,3$, the constitutive equations correspond to a second grade fluid. It has been shown  that the Clausius-Duhem
				inequality and the assumption that the Helmholtz free energy   is minimal at equilibrium requires the material moduli to satisfy
				\begin{align}\label{secondlaw}
				\nu \geq 0,\quad \alpha_1+\alpha_2=0, \quad \alpha_1\geq 0. 
				\end{align}

	Although second grade fluids are mathematically more treatable,	dealing  with  several non-Newtonian fluids, the rheologists have not
				confirmed these restrictions \eqref{secondlaw}, thus give  the conclusion that the fluids that have been tested are not fluids of second grade but are fluids that are characterized by a different constitutive structure, we refer to \cite{FR80} and references therin for more details. Moreover,  the second grade fluid model does not capture important rheological properties as for instance the shear thinning and shear thickening effects, so there is a real need to study the more complex third	grade fluid model. Following \cite{FR80}, in order to allow the motion of the fluid to be compatible with thermodynamic, it should be imposed that
						
				\begin{equation*}
				\nu \geq 0, \quad \alpha_1\geq 0, \quad |\alpha_1+\alpha_2 |\leq \sqrt{24\nu\beta}, \quad \beta_1=\beta_2=0, \beta_3=\beta \geq 0.
				\end{equation*}

From a practical point of view, recently special attention has been devoted to the study of non-Newtonian viscoelastic fluids of differential type. It is worth to mention that several simulations studies have been performed by using the third grade fluid models, in order to understand and explain the characteristics of several nanofluids (see 
\cite{ HUAal20, PP19,RHK18} and references therein).
We recall that nanofluids  are  engineered colloidal suspensions of nanoparticles (typically made of metals, oxides, carbides, or carbon nanotubes) in a base fluid as water, ethylene glycol and oil, which exhibit enhanced thermal conductivity compared to the base fluid, which turns out to be of  great  potential to be used in  technology, including heat transfer, microelectronics, fuel cells, pharmaceutical processes, hybrid-powered engines, engine cooling/vehicle thermal management, etc.  
Therefore  the  mathematical analysis of third grade fluids equations should be relevant to  predict and control the  behavior of these fluids, in order  to design optimal flows that can be successfully used and applied in the industry. From mathematical point of view,  fluids of grade $3$ constitute an hierarchy of fluids with  increasing complexity and more  nonlinear terms, which is  more complex  and require more involved analysis.\\
\\

The study the equation \eqref{111} requires a boundary condition, and an initial condition in a suitable functional space. 
Besides the most studies on fluid dynamic equations consider the Dirichlet  boundary condition, which assumes that the particles adjacent to the boundary surface have the same velocity  as the boundary, 
there are physical reasons to consider slip boundary conditions. Namely,   practical studies (see e.g  \cite{ RKWA17, WD97})
show that  viscoelastic fluids slip against  the boundary, and on the other hand, mathematical studies turn out that the Navier boundary conditions are compatible with the vanishing viscosity transition (see \cite{CC_1_13, CC_2_13, CMR98, K06}). 
We recall that for  appropriate nondimensionalizations, the Reynolds number $Re$ is equal to $1/\nu$, then the vanishing viscosity corresponds to the 
transition to turbulent regime, which is associated with high values of $Re$. In this paper we consider  a homogeneous Navier-slip boundary condition.
Let us mention that the third grade fluid equation  with the Dirichlet boundary conditions was studied in  \cite{AC97, SV95}, where the authors  proved the existence and the uniqueness of local solutions for initial conditions in $H^3$ or  global in time solution for small initial data when compared with the viscosity (see also \cite{BL99}). 
Later on \cite {Bus-Ift-1, Bus-Ift-2}, the authors considered the  equation with a homogeneous Navier-slip boundary conditions and established the well-posedness of a global solution for initial conditions in $H^2$, without any restriction on the size of the data. Recently, the authors in \cite{AC20, Cip-Did-Gue} extended  the later deterministic result  to  the stochastic models.  It is worth recalling that the question of uniqueness in 3D is always an open problem.\\

 The control problems  of Newtonian fluids  (fluids of grade $1$), where the flows are   described by  Navier-Stokes  equations have been extensively studied in the literature. In general, the issue was the  control of  the  turbulence inside a flow  or tracking the velocity of the flows. Without exhaustiveness, let us refer to   \cite{Abergel90,Delos-Griesse, Hinze-Kunisch} and the references therein. Directing  a  velocity field to a desired velocity
	field over time has a wide range of applications in engineering and science such as   combustion, chemical reacting flows and  design
	problems $\cdots$etc. (see e.g. \cite{Gunzburger99}). There is a large literature on tracking control problems for Newtonian flows. Let us mention  
	 \cite{Gunzburger99}, where the authors   derived an optimality system for the optimal solutions for an optimal control problem of tracking the velocity for Navier-Stokes flows in bounded two-dimensional domains with bounded distributed controls.  Then, a second-order sufficient optimality condition were established in \cite{Tro-Wach06}. In \cite{Chem-Cip}, the authors considered  a boundary optimal control for two dimensional Navier-Stokes equations, where  the control acts on the boundary through an injection-suction device. Recently, the authors in   \cite{CASAS-KUNISCH} studied an optimal control problem for a two dimensional Navier-Stokes equations with measure valued controls.  \\

On the other hand, despite that there exist many real industrial or physiological fluids   cannot be described by the classical linearly viscous Newtonian model,  the optimal control of non-Newtonian viscoelastic fluids have been less considered and rarely investigated. To the best of the author’s knowledge, the first result in the theory of optimal control of viscoelastic (non-Newtonian) fluids have been achieved in \cite{Kunisch2000}. The authors  studied an optimal control of viscoelastic fluid flow in a $4$ to $1$ contracting channel, where  the control mechanism is based on	heating or cooling the fluid along a portion of the boundary of the flow domain.  They obtained an  optimality system, derived by
	the use of the Lagrange multipliers with two different cost functionals: one of tracking type, with the viscoelastic flow being tracked to the flow
	corresponding to the viscous flow, the other penalizing negative contributions of the velocity component in direction of the span of the
	channel. Some numerical simulations was considered as well. In \cite{Wach-Roubi2010},   an optimal control problem for the evolutionary flow for incompressible non-Newtonian fluids in  a two-dimensional bounded domain have been studied. The authors proved   the existence of optimal controls  and established   first-order necessary  and second-order sufficient  optimality conditions, where the cost functional was of tracking type.  In \cite{Arada14},  the author considered an  evolutionary flow of incompressible quasi-Newtonian  shear-thickening fluids in two and three dimensional setting. The   Newtonian constitutive equation incorporating a shear-rate-dependent viscosity, where the viscosity  increases with increasing shear rate such as Carreau model (see  \cite[(1.1)]{Arada14}). He studied the  control  of the system through a distributed mechanical force leading the velocity to a given target field and established  a necessary optimality conditions.\\

It is worth mentioning that the previous works considered non-Newtonian fluids but not of differential type. To the best of our knowledge, the control problem for the second grade fluids (differential type) has been adressed for the first time in \cite{Arada-Cipriano}, where the authors proved the existence of an optimal control and deduced the first order optimality conditions, where the cost functional was of tracking type. Recently, the authors in  
\cite{CC22} established an uniqueness result for the complete  first order optimality system, by assuming enough intensity of the cost. 
The control problem for stochastic second grade fluid models have been studied in \cite{CC18, CP19}.  Since there are many applications  contain
 control mechanisms that one would like to adjust in an optimal way to achieve a given
objective as well as possible, this work corresponds to the next step  to control complex differential fluids. As far as we know, the optimal control problem for third grade fluids is being adressed  here for the first time. \\

The article is organized as follows:
in Section \ref{S3}, we state  the third grade fluid model and define the  appropriate functional spaces. Then, we  collect some  estimates for the state already available in the literature and that are convenient for our analysis. Next, we formulate the control prblem and  establish the main results of the article.
Section \ref{S4} is devoted to show the existence and the uniqueness of the solution to the linearized state equation.
In Section \ref{S5}, we prove a stability result for the state equation, which will be a key ingredient in Section \ref{S6} to  study the differentiability of the control-to-state mapping.
In Section \ref{S7}, we write the adjoint equations  and prove the existence and uniqueness of the solution. 
Finally, in  Section \ref{S8} we   establish a duality relation between the solution of the linearized equation and the adjoint state.  Next, we prove the existence of the solution to the control problem, and we deduce the first order optimality  condition.  Section \ref{S-uniq}, which deals with the quadratic Lagrangian,  is devoted to the proof of the uniqueness of the solution to the coupled system for a large cost intensity. 

\section{Formulation of the control problem and main results}\label{S3}
In this  section, we present some results known in the literature that will be  convenient for further analysis. Next, we formulate the  control problem and establish the  main results.
\subsection{The state equation}
The goal of this work is to study the optimal  control of a  non-Newtonian third grade fluid, where the control is introduced via the  external forces. The fluid  fills a bounded  and simply connected domain $D  \subset  \mathbb{R}^2$ with regular (smooth) boundary $\partial D$, and it is governed by the following equations
\begin{equation}
\label{I}
\left\{\begin{array}{ll}
\partial_t(v(y))-\nu \Delta y+(y\cdot \nabla)v+\sum_{j}v^j\nabla y^j-(\alpha_1+\alpha_2)\text{div}(A^2) -\beta \text{div}(|A|^2A)&\\[0.5mm]
\qquad= -\nabla \mathbf{P} +U, \quad\quad \text{div}(y)=0 \quad &\hspace*{-3cm}\text{in } D\times (0,T) ,\\[0.15cm]
y\cdot \eta=0, \quad [\eta \cdot  \mathbb{D}(y)]\cdot \tau=0  \quad &\hspace*{-3cm}\text{on } \partial D \times (0,T),\\[0.15cm]
y(x,0)=y_0(x) \quad & \hspace*{-3cm}\text{in } D,
\end{array}
\right.
\end{equation}
where $y:=(y_1,y_2)$ denotes  the fluid velocity field,  $v:=v(y):=y-\alpha_1 \Delta y$ and   $A=A(y)=\nabla y+\nabla y^T=2 \mathbb{D}(y)$. $\mathbf{P}$ denotes  the pressure and $U=(U_1,U_2)$ denotes the external force. The pair  $(\eta,\tau)$
stands for  the  external  normal vector and the unitary tangent vector to the boundary $\partial D$ with positive  orientation.
The constants  $\alpha_1,\alpha_2$ and $\beta$ are  the material moduli and satisfy the   conditions:
\begin{equation}
\label{condition1}
\nu \geq 0, \quad \alpha_1\geq 0, \quad |\alpha_1+\alpha_2 |\leq \sqrt{24\nu\beta}, \quad  \beta \geq 0. 
\end{equation}
It is worth recalling that  \eqref{condition1} allows the motion of the fluid to be compatible with  thermodynamic laws. 
\\

\subsection{Functional spaces and notations}
For a functional space $E$ and a positive integer $k$, we define
$$  (E)^k:=\{(f_1,\cdots,f_k): f_l\in E,\quad l=1,\cdots,k\}.$$
Let us introduce the following spaces:
\begin{equation}
\begin{array}{ll}
H&=\{ y \in (L^2(D))^2 \,\vert \text{ div}(y)=0 \text{ in } D \text{ and } y\cdot \eta =0 \text{ on } \partial D\}, \nonumber\\[1mm]
V&=\{ y \in (H^1(D))^2 \,\vert \text{ div}(y)=0 \text{ in } D \text{ and } y\cdot \eta =0 \text{ on } \partial D\}, \\[1mm]
W&=\{ y \in V\cap (H^2(D))^2\; \vert\, (\eta \cdot 
\mathbb{D}(y))\cdot \tau =0 \text{ on } \partial D\}, \nonumber\\[1mm]
\widetilde{W}&=(H^3(D))^2\cap W. \nonumber
\end{array}
\end{equation}

First, we recall the Leray-Helmholtz projector $\mathbb{P}: (L^2(D))^2 \to H$, which is a linear bounded operator characterized by the following $L^2$-orthogonal decomposition
$$v=\mathbb{P}v+\nabla \varphi, \quad \forall \varphi \in H^1(D). $$

Now, let us introduce  the scalar product between two matrices  $
A:B=tr(AB^T)$
and denote $\vert A\vert^2:=A:A.$
The divergence of a  matrix $A\in \mathcal{M}_{2\times 2}(E)$ is given by 
$(\text{div}(A)_i)_{i=1,2}=(\displaystyle\sum_{j=1}^2\partial_ja_{ij})_{i=1,2}. $\\

The space $H$ is endowed with the  $L^2$-inner product $(\cdot,\cdot)$ and the associated norm $\Vert \cdot\Vert_{2}$. We recall that
\begin{align*}
(u,v)&=\int_D u\cdot vdx=\sum_{i=1}^2\int_Du_iv_idx, \quad  \forall u,v \in (L^2(D))^2,\\
(A,B)&=\int_D A: Bdx , \quad  \forall A,B \in \mathcal{M}_{2\times 2}(L^2(D)).
\end{align*}
On the functional spaces $V$, $W$ and  $\widetilde{W}$,
we will consider  the following inner products 
\begin{equation}
\begin{array}{ll} 
(u,z)_V&:=(v(u),z)=(u,z)+2\alpha_1(\mathbb{D}u,\mathbb{D}z),\\[1mm]
(u,z)_W&:=(u,z)_V+(\mathbb{P}v(u),\mathbb{P}v(z)),\\[1mm]
(u,z)_{\widetilde{W}}&:=(u,z)_V+(\text{curl}v(u),\text{curl}v(z)),
\end{array}
\end{equation}
and denote by $\Vert \cdot\Vert_V,\Vert \cdot\Vert_W$ and $\Vert \cdot\Vert_{\widetilde{W}}$ the corresponding norms.\\

For the sake of simplicity, we do not distinguish between scalar, vector or matrix-valued   notations when it is clear from the context. In particular, $\Vert \cdot \Vert_E$  should be understood as following
\begin{itemize}
	\item $\Vert f\Vert_E^2= \Vert f_1\Vert_E^2+\Vert f_2\Vert_E^2$ for any $f=(f_1,f_2) \in (E
	)^2$.
	\item $\Vert f\Vert_{E}^2= \displaystyle\sum_{i,j=1}^2\Vert f_{ij}\Vert_E^2$ for any $f\in \mathcal{M}_{2\times 2}(E)$.
\end{itemize}

Throughout the article, we will denote by $C,C_i, i\in \mathbb{N}$,   generic constants, which may varies from line to line.\\


Now, let us recall   some results about  the solution of \eqref{I}  based on  \cite{Bus-Ift-2}.
\begin{theorem}\cite[Thm. 1]{Bus-Ift-2}
	\label{THm1} Let $y_0\in W$ and $U\in L^2_{loc}([0,\infty[;(L^2(D))^2)$, then there exists a global  unique solution $y$ to \eqref{I}. Moreover, if $y_0\in \widetilde{W}$ and  $U\in L^2_{loc}([0,\infty[;(H^1(D))^2)$  the solution $y$ belongs to $L^\infty_{loc}([0,\infty[;\widetilde{W})$.
\end{theorem}
Following \cite{Bus-Ift-2}, for $t\geq 0$ we have  the following results.
\begin{lemma}($H^1$  estimates)  There exist $C,C_1,C_2,K_0>0$ such that 
\begin{align*}
	\Vert y(t)\Vert_{H^1} &\leq  e^{Ct}\big(\Vert y_0\Vert_{H^{1}}+\Vert U\Vert_{L^2(D\times (0,t))}\big):=M_0(t);\\[1mm]
	\Vert A(y)\Vert_{L^4(D\times (0,t))} &\leq \bigl(\dfrac{2}{\beta}\bigr)^{\frac{1}{4}}\sqrt{M_0(t)};\\
	\Vert y\Vert_{L^4((0,t); W^{1,4}(D))}&\leq  K_0e^{\frac{Ct}{2}}\bigl[\bigl(\dfrac{2}{\beta}\bigr)^{\frac{1}{4}}+C_1t^{\frac{1}{4}}e^{\frac{Ct}{2}}\bigr]\left(1+\Vert y_0\Vert_{H^{1}}+\Vert U\Vert_{L^2(D\times (0,t))}\right):=M_1(t);\\[0.5mm]
	\Vert y \Vert_{L^1(0,t;L^\infty(D))} &\leq C_2t^{\frac{3}{4}}M_1(t):=M_2(t).
	\end{align*}
\end{lemma}
\begin{lemma}($H^2$ estimates)  There exists $C_3>0$ such that 
	\begin{align*}
	\Vert y(t) \Vert_{H^2}^2&+\min\bigl(\beta,\dfrac{\alpha_1\beta}{4}\bigr)\int_0^t\Vert \vert A(y)\vert^2\Vert_{H^{1}}^2ds+\dfrac{\alpha_1\beta}{2}\int_0^t\int_D\vert A(y)\vert^2\vert \nabla A\vert^2dxds\\
	& \leq e^{C_3(t+M_2(t)+tM_0^2(t))}\Big(\Vert y_0\Vert_{H^2}+2\int_0^t\Vert U(s)\Vert_2^2ds+C_3M_1^4(t)\Big) :=M_3(t). 
	\end{align*}
\end{lemma}
\begin{lemma}\label{Lemma-H3}($H^3$  estimates)  Let $t\geq 0$, then there exists $\epsilon_0(t):=\epsilon_0(t,\Vert U\Vert_{L^2(0,t;H^1(D))})>0$ such that 
	\begin{align}\label{state-H3-estimate}
	\Vert y(t) \Vert_{H^3}^2 \leq 4^{\dfrac{1}{\epsilon_0(t)}}\big(1+ \Vert y_0\Vert_{H^{3}}^2\big):=M_4^2(t).
	\end{align}
\end{lemma}
\begin{remark}It is worth noting that Theorem \ref{THm1} holds with no   smallness conditions on the data, where Navier slip boundary conditions play a crucial role to prove Theorem \ref{THm1}. Indeed,  $\eqref{I}_3$ ensures that $\Delta y$ is almost	tangent to the boundary, in the sense that it can be expressed in terms of derivatives of	order $1$ of $y$ (see \cite[Prop. 2]{Busuioc}). Thanks to the last fact, we can perform some integrations by parts, which yields some boundary terms which can estimated in a satisfactory manner. More precisely, one can prove  that the pressure term vanishes and then  $H^2$-estimates for $y$ was obtained. In contrast with homogeneous Dirichlet boundary condition, where $\Delta y$ does not enjoy such property on the boundary  and then Theorem \ref{THm1} does not holds with homogeneous Dirichlet boundary condition. Concerning third grade fluids equation on bounded domain with homogeneous Dirichlet boundary condition, one can prove only local existence and uniqueness for large data and    global existence and uniqueness of solutions for small initial data in $H^3$ in 2D case ( see e.g. \cite{AC97});   the global existence of solutions for large data in this case remains an open problem. Finally,  in the  2D case, the $H^3$ regularity is shown to be propagated by the equation by using some Sobolev embeeding inequalities, which are true only in the 2D setting, we refer to  \cite{Bus-Ift-2} for more details. 		
	\end{remark}
Now we introduce the following modified Stokes problem (see \cite{ CC18,S73} for the properties of the solution)
\begin{equation}
\label{Stokes}
\left\{\begin{array}{ll}
h-\alpha_1\Delta h+\nabla \pi=f, \quad 
\text{div}(h)=0 \quad &\text{in } D,\\[2mm]
h\cdot \eta=0, \quad [\eta \cdot \mathbb{D}(h)]\cdot \tau=0  \quad &\text{on } \partial D,	
\end{array}
\right.
\end{equation}
and denotes its solution  $h$  by $h=(I-\alpha_1\mathbb{P}\Delta)^{-1}f$.
We also consider the trilinear form 
$$
b(\phi,z,y)=(\phi\cdot \nabla z,y)=\int_D(\phi\cdot \nabla z)\cdot y dx, \quad \forall \phi,z,y \in (H^1(D))^2,$$
which verifies $b(y,z,\phi)=-b(y,\phi,z),\quad\forall y \in V; \forall z,\phi \in H^1(D)$.
\vspace{2mm}\\
Let $T >0$,  we assume that the initial data $y_0$ and the force $U$ satisfy 
\begin{align}\label{H1}\tag{H-1}
y_0\in \widetilde{W},\quad U\in L^2(0,T;(H^1(D))^2)	.
\end{align}
\subsection{Control problem}

Our main goal is to control the solution of the equation \eqref{I} 
by a distributed force $U$.
The control variables $U$ belong to the set $\mathcal{U}_{ad}$ of  admissible controls, which is defined as a  nonempty bounded closed convex subset of $L^2(0,T;(H^1(D))^2)$.  In other words
$$ \mathcal{U}_{ad}:=\{ u \in L^2(0,T;(H^1(D))^2):\quad \Vert u \Vert_{L^2(0,T;(H^1(D))^2)}\leq K \}; \quad  0<K<\infty. $$
We consider the cost functional   given by
\begin{align}\label{cost-uniqueness}
J(u,y)=\dfrac{1}{2}\int_0^T\Vert y-y_d\Vert_2^2dt+\dfrac{\lambda}{2} \int_0^T\Vert u\Vert_2^2dt,
\end{align}
where  $y_d \in L^2(D\times (0,T))$ corresponds to  a desired target field and  $\lambda \geq 0$.
 The control problem reads
\begin{align}\label{Problem-uniq}
\displaystyle\min_{u \in \mathcal{U}_{ad}}
\bigg\{ \dfrac{1}{2}\int_0^T\Vert y-y_d\Vert_2^2dt+\dfrac{\lambda}{2} \int_0^T\Vert u\Vert_2^2dt: \; y \text{ is the solution of \eqref{I}
	with  force }  u     \bigg\}.
\end{align}

	\begin{remark}
		We wish to draw the reader’s attention  that
	  we can consider more general class of   cost functional given by
		\begin{align*}
		J(U,y)=\int_0^TL(t,U(t),y(t))dt,
		\end{align*}
		where the Lagrangian  $L:[0,T]\times (H^1(D))^2\times \widetilde{W} \to \mathbb{R}^+$  satisfies   some properties ( see  e.g. \cite{CC18}).
	\end{remark}
\begin{remark}
	\begin{enumerate}
		\item 
		In this article,  the Lagrangian $L$ is given by
$L(\cdot,u,y)=\dfrac{1}{2}\Vert y-y_d\Vert_{2}^2+\dfrac{\lambda}{2}\Vert u\Vert_{2}^2. $
Therefore, we have 
\begin{align*}
	\int_0^T(\nabla_yL(t, U(t),y(t)),v(t))dt&=\int_0^T(v(t),y(t)-y_d(t))dt ,\\ \int_0^T(\nabla_uL(t, U(t),y(t)),v(t))dt&=\lambda\int_0^T(v(t),U(t))dt,
\end{align*}
 for any $v\in L^2(0,T;(L^2(D))^2)$.
\\
\item  We wish to draw the reader’s attention to the fact that  \eqref{cost-uniqueness} is well defined for $u \in L^2(0,T;(L^2(D))^2)$ but our aim is to solve \eqref{Problem-uniq} and establish an optimality condition. The $H^3$-regularity of the solution to the state equation \eqref{I} play a crucial role in the analysis of the linearized and adjoint equations (see Sect.\ref{Linearized-section} and Sect.\ref{S7}). Moreover, it is also crucial to establish the necessary optimality condition (see Sect.\ref{S6}). For that, we consider $\mathcal{U}_{ad}$ as a subset of $L^2(0,T;(H^1(D))^2),$ which ensures the $H^3$-regularity of the solution to  \eqref{I}, see Lemma \ref{Lemma-H3}.
\end{enumerate}
\end{remark}

\begin{remark}
	Under \eqref{H1},  the solution $y$ of \eqref{I}, belongs to $L^\infty(0,T;\widetilde{W})$.
\end{remark}
\subsection{Main results} Our first main result shows the existence  of a solution to the control problem, and  establishes the first order optimality conditions. 
\begin{theorem}\label{main-thm} 
	Assume \eqref{H1}. Then the control problem \eqref{Problem-uniq}  admits, at least, one optimal solution 
	$$(\tilde{U},\tilde{y}) \in \mathcal{U}_{ad} \times  \big(L^\infty(0,T;\widetilde{W})) \cap H^1(0,T;V) \big ), $$ 
	where $\tilde{y}$ is the unique solution of \eqref{I} with $U=\tilde{U}$. Moreover, there exists a unique solution  $\tilde{p}$ of \eqref{adjoint} with $f=\nabla_yL(\cdot, \tilde{U},\tilde{y})$, such that if $\tilde{z}$  is the solution of \eqref{Linearized} for $y=\tilde{y}$ and $\psi=\psi-\tilde{U}$,  the following duality property 
	$$\int_0^T(\psi(t)-\tilde{U}(t),\tilde p(t))dt=\int_0^T(\nabla_yL(t,\tilde U(t), \tilde y(t)) ,\tilde z(t))dt, $$	and the following optimality condition 	hold
	\begin{equation}
	\label{oocc}
	\int_0^T(\psi(t)-\tilde{U}(t),\tilde p(t)+\nabla_uL(t,\tilde U(t), \tilde y(t)))dt \geq 0.
	\end{equation}
\end{theorem}
An additional  step in the study of the control problem  relies on the analysis of the solutions of the coupled system  constituted by the state equation \eqref{I}, the adjoint equation \eqref{adjoint}  and the optimality relation \eqref{oocc}.
Our next result goes in this direction and establishes an uniqueness result for the solutions of the coupled system for \eqref{Problem-uniq}.
\begin{theorem}\label{Thm-uniq}
	Assume that  $\lambda >2 \widetilde{C}\widetilde{\lambda}(\Gamma+4\kappa(\alpha_1+\alpha_2)+12\kappa\beta\gamma )$, where $\gamma=\sup_{t\in [0, T]}M_4(t)$ (see \eqref{state-H3-estimate}) and $\widetilde{\lambda},\widetilde{C}, \kappa$ are given by \eqref{estimate-adjoint}, \eqref{stability-estimate}, \eqref{embedding-constant-uniq}, respectively.
	Then 	the optimal control problem \eqref{Problem-uniq} has a unique global solution. 
\end{theorem}
\begin{remark}
	\begin{enumerate}
\item 	It is important to underline that Theorem \ref{Thm-uniq} is achieved under a natural condition $\lambda >2 \widetilde{C}\widetilde{\lambda}(\Gamma+4\kappa(\alpha_1+\alpha_2)+12\kappa\beta\gamma )$, which gathers the size of the initial data, the parameters of the model  $(\nu, \alpha_1,\alpha_2,\beta)$ and   the intensity
	of the cost $\lambda$. In other words, if the fluid material is sufficiently viscous and elastic and
	the initial condition is small enough, or instead if the intensity of the cost is big
	enough, the solution of the first-order optimality system is unique, and corresponds
	to the unique  solution of the optimal control problem.\\

\item 	The standard approach in optimization problem is based on the analysis of the second order sufficient condition, which requires a coercivity of the Lagrange function combined with the first order necessary optimality condition. Generally, the second order optimality condition cannot be expected if $\lambda=0$ (see e.g. \cite[Rmq. 3.16]{Tro-Wach06}). Following a similair analysis as in \cite[Subsect. 3.1]{CC22}, one can get a second-order G\^ateaux derivative  of the control-to-state-mapping, by taking into account the additional terms of the third grade fluids model. Note that  an analysis of the second-order G\^ateaux derivative  of the control-to-state-mapping will require  more regularity with respect to the space variable $x$ of the solution of linearized equation $z$ (see e.g. \cite[Sect. 4 ]{Wach-Roubi2010} for similar issues). Additionally,
	  one can  found, in the literature, that the coercivity of the Lagrange function leads to some condition (see e.g.  \cite[Th. 3.2]{CC22}), which  mimics  $\lambda >2 \widetilde{C}\widetilde{\lambda}(\Gamma+4\kappa(\alpha_1+\alpha_2)+12\kappa\beta\gamma )$.   On the other hand, by using \eqref{estimate-adjoint} one can see that the obove condition leads to  
	  \begin{align}\label{control-parameter}
	  	\lambda >\widetilde{C}C(T)(\Gamma+4\kappa(\alpha_1+\alpha_2)+12\kappa\beta\gamma )\Vert y-y_d\Vert_{L^2(Q)}.
	  \end{align}
	  \eqref{control-parameter} means that the desired state $y_d$ could be approximated closely if $\lambda$ is  large enough or $\Vert y-y_d\Vert_{L^2(Q)} $ is small enough,
which is  similair  to  the standard  second-order optimality condition related to  the control theory of  Navier-Stokes equations (see e.g. \cite[Eqn. 3.30]{Tro-Wach06}).  Finally, we  consider here the uniqueness of global system, since it does not request more regularity for $z$ and we believe that the analysis of second order optimality condition will lead to a  similar condition.
\end{enumerate}
\end{remark}
The proof of Theorem \ref{main-thm}  will be  splitted  into several steps. First, we prove the solvability of the linearized state equation. Then, we establish  a stability result for the state, which is a key ingredient to show that the solution of the linearized state equation corresponds to the 
G\^ateaux derivative of the control-to-state mapping. Next,  we will write and study the well-posedness of the adjoint equation. Finally,
we prove  the existence of an optimal pair and deduce the first order  optimality condition. The proof of Theorem \ref{Thm-uniq} is presented in Section \ref{S-uniq}.
\section{Linearized state equation}\label{Linearized-section}
\label{S4}
This section is devoted to the study of the linearized state equation.
The existence of the solution is based on the  Faedo-Galerkin's approximation method, which relies on a special basis  designed according to the structure of the equation, in order to derive the uniform estimates in $H^1$ and in $H^2$.
\vspace{2mm}\\
Let us consider $\psi:D\times ]0,T[ \to \mathbb{R}^2$ such that
\begin{align}\label{psi}
\psi\in (L^2(D\times]0,T[))^2.
\end{align}
Our goal is to prove an existence and uniqueness result for the following problem
\begin{align}
\label{Linearized}
\left\{\begin{array}{ll}
\partial_t(v(z))-\nu \Delta z+(y\cdot \nabla )v(z)+(z\cdot \nabla)v(y)+\sum_{j}v(z)^j\nabla y^j&\\[2mm]
\quad+\sum_{j}v(y)^j\nabla z^j-(\alpha_1+\alpha_2)\text{div}[A(y)A(z)+A(z)A(y)]&\vspace{2mm}\\
\quad -\beta\text{ div} \big[ |A(y)|^2A(z)\big]
-2\beta\text{ div}\left[ \left( A(z):A(y)\right)A(y)\right]=\psi-\nabla \pi \quad &\hspace*{-0.4cm}\text{in } D\times (0,T), \\[2mm]
\text{div}(z)=0 \quad &\hspace*{-0.4cm}\text{in } D\times (0,T),\\[2mm]
z\cdot \eta=0 \quad [\eta \cdot \mathbb{D}(z)]\cdot \tau=0  \quad &\hspace*{-0.4cm}\text{on } \partial D\times (0,T),\\[2mm]
z(x,0)=0 \quad &\hspace*{-0.4cm} \text{in } D.
\end{array}
\right.
\end{align}
\begin{definition}\label{Def-lin}
	A function $z \in L^{\infty}(0,T;W)$ with $\partial_tz \in L^{2}(0,T;V)$ is a solution of \eqref{Linearized} if $z(0)=0$ and  for any $t\in [0,T]$, the following equality holds
	\begin{equation*}
	\begin{array}{ll}
	&(\partial_tv(z),\phi)+2\nu(\mathbb{D} z,\mathbb{D}\phi)+b(y,v(z),\phi)+b(z,v(y),\phi)+b(\phi,y,v(z))+b(\phi,z,v(y))\vspace{1mm}\\[1mm]
	&\quad+(\alpha_1+\alpha_2)\big(A(y)A(z)+A(z)A(y),\nabla \phi\big) 
	+\beta\big( |A(y)|^2A(z), \nabla\phi\big)\vspace{1mm}\\[1mm]&\qquad+2\beta \big((A(z):
A(y))A(y),\nabla\phi\big)= (\psi,\phi),
	\qquad \text{for all } \phi \in W.
	\end{array}
	\end{equation*}
\end{definition}
\begin{remark}
	For $u \in L^{\infty}(0,T;W)$ and  $\partial_tu \in L^{2}(0,T;V)$, 
	$(\partial_tv(u),\phi)$ will be understood in the following sense
	$$(\partial_tv(u),\phi)=(\partial_tu,\phi)_V=(\partial_tu,\phi)+2\alpha_1(\partial_t \mathbb{D}u,\mathbb{D} \phi). $$
\end{remark}
\subsection{Approximation}
Following the same strategy as in \cite{CC18}, the  solution of \eqref{Linearized}  can be obtained as a limit of the finite dimensional Faedo-Galerkin's approximations. In this way, let us consider 
an orthonormal basis $\{h_i\}_{i\in \mathbb{N}} \subset H^4(D) \cap W$  in $V$, which  satisfies
\begin{align}\label{basis1}
(v,h_i)_W=\mu_i(v,h_i)_V, \quad \forall v \in W, \quad i \in \mathbb{N},
\end{align}
where the sequence $\{\mu_i\}$ of the corresponding eigenvalues fulfils the  properties:\\ $\mu_i >0, \forall i\in \mathbb{N},$ and $\mu_i \to \infty$ as $i \to \infty$. As a consequence of \eqref{basis1}, the sequence $\{\tilde{h}_i=\frac{1}{\sqrt{\mu_i}}h_i\}$ is an orthonormal basis in $W$. Let us introduce the Galerkin approximations of \eqref{Linearized}.\\ Consider $W_n=\text{span}\{h_1,\cdots,h_n\}$ and define
$ z_n(t)=\sum_{i=1}^n c_i(t)h_i$   for each  $t\in [0,T].$
The approximated problem for \eqref{Linearized} reads $z_n(0)=0$ and 
\begin{align}\label{approximation}
\begin{cases}
(\partial_tv(z_n),\phi)= \Big(\psi+\nu \Delta z_n-(y\cdot \nabla )v(z_n)-(z_n\cdot \nabla)v(y)-\sum_{j}v(z_n)^j\nabla y^j&\\\quad-\sum_{j}v(y)^j\nabla z_n^j
+(\alpha_1+\alpha_2)\text{div}[A(y)A(z_n)+A(z_n)A(y)] &\\\quad+\beta\text{ div}\big [ |A(y)|^2A(z_n)\big]+2\beta\text{ div} \big [(A(z_n):A(y))A(y)\big],\phi\Big),    \quad \text{ for any }\; \phi \in W_n. 
\end{cases}
\end{align}
Note that \eqref{approximation} defines  a system of  linear ordinary differential equations, which has a unique  solution 
\begin{align}\label{local-time-lin}
z_n\in \mathcal{C}([0,T_n],W_n).
\end{align}
\subsection{Uniform  estimates}
\subsubsection{Estimate in the space $V$ for $z_n$}
Setting $\phi=h_i$ in \eqref{approximation}, we have
\begin{equation}
\begin{array}{ll}
\label{approximationz-n}
(\partial_tv(z_n),h_i)&=\Big(\psi+\nu \Delta z_n-(y\cdot \nabla )v(z_n)-(z_n\cdot \nabla)v(y)-\sum_{j}v(z_n)^j\nabla y^j\vspace{1mm}\\
&-\sum_{j}v(y)^j\nabla z_n^j+(\alpha_1+\alpha_2)\text{div}[A(y)A(z_n)+A(z_n)A(y)] \vspace{1mm}\\&+\beta\text{ div}\big [ |A(y)|^2A(z_n)\big]
+2\beta\text{ div} \left [\left(A(z_n):A(y)\right)A(y)\right],h_i\Big):=(f(z_n),h_i).
\end{array}
\end{equation}
Multiplying the equality by $c_i(t)$ and summing from $i=1$ to $n$, we deduce
\begin{align}\label{estimateV}
2(\partial_t v(z_n),z_n) =2(f(z_n),z_n).
\end{align}
Integrating by parts and using the boundary conditions, we obtain 
\begin{align*}
2(\partial_t v(z_n),z_n)&=2(\partial_t (z_n),z_n)+4\alpha_1(\partial_t \mathbb{D}z_n,\mathbb{D}z_n)=\dfrac{d}{dt}\big(\Vert z_n\Vert_2^2+2\alpha_1\Vert  \mathbb{D}z_n\Vert_2^2\big).
\end{align*}
Let us estimate the term $(f(z_n),z_n)$.
\begin{equation*}
\begin{array}{ll}
I_1&=2\Big(\psi+\nu \Delta z_n-(y\cdot \nabla )v(z_n)-(z_n\cdot \nabla)v(y)-\sum_{j}v(z_n)^j\nabla y^j-\sum_{j}v(y)^j\nabla z_n^j\vspace{1mm}\\
&+(\alpha_1+\alpha_2)\text{div}[A(y)A(z_n)+A(z_n)A(y)] +\beta\text{ div}\big [ |A(y)|^2A(z_n)\big]\vspace{1mm}\\
&+2\beta\text{ div} 
\big [\left(A(z_n):A(y)\right)A(y)\big],z_n\Big)=I_1^0+I_1^1+I_1^2+I_1^3+I_1^4.
\end{array}
\end{equation*}
Thanks to the free divergence property, we have
\begin{align*}
I_1^0=2\big(\psi +\nu \Delta z_n,z_n\big)=2\big(\psi,z_n\big)-4\nu\Vert \mathbb{D} z_n\Vert_2^2.
\end{align*}
Notice that
\begin{equation*}
\begin{array}{ll}
I_1^1
=-2\big(b(y,v(z_n),z_n)+b(z_n,y,v(z_n))+b(z_n,v(y),z_n)+b(z_n,z_n,v(y))\big).
\end{array}
\end{equation*}
Using \cite[Lem. 3.5.]{CC18} we write 
\begin{equation*}
\begin{array}{ll}
I_1^1
&=2\big( b(y,z_n,v(z_n))-b(z_n,y,v(z_n))\big)=-2(\text{curl}(v(z_n))\times y,z_n),\vspace{2mm}\\
|I_1^1|&=|2\big(\text{curl}(v(z_n))\times y,z_n\big)|\leq C_1\Vert y\Vert_{\widetilde{W}}\Vert z_n\Vert_{V}^2.
\end{array}
\end{equation*}
The Stokes theorem allows to  infer that
\begin{align*}
I_1^2&=2(\alpha_1+\alpha_2)\big(\text{div}[A(y)A(z_n)+A(z_n)A(y)],z_n\big)\\[1.5mm]
&=-2(\alpha_1+\alpha_2)[\int_D [A(y)A(z_n)+A(z_n)A(y)]:\nabla z_n dx +\int_{\partial D}(A(y)A(z_n)\eta+A(z_n)A(y)\eta)\cdot  z_n]dS.
\end{align*}
Since $y$ and $z_n$ satisfy the Navier  boundary conditions (see $\eqref{Linearized}_3$), we have
$$ \int_{\partial D}(A(y)A(z_n)\eta+A(z_n)A(y)\eta)\cdot  z_n dS=0.$$
Hence, H\"older inequality ensures
\begin{align*}
\vert I_1^2\vert &=\left\vert 2(\alpha_1+\alpha_2)\int_D [A(y)A(z_n)+A(z_n)A(y)]:\nabla z_n dx\right\vert
\vspace{1mm} \\
&\leq  C_2  \Vert z_n\Vert_{H^1}^2\Vert  y\Vert_{W^{1,\infty}}\leq C_2  \Vert z_n\Vert_{V}^2\Vert  y\Vert_{\widetilde{W}}.
\end{align*}
Applying the Stokes theorem once more and using the symmetry of $A(z_n)$, we derive
\begin{align*}
I_1^3&= 2\beta\big(\text{ div} \big ( |A(y)|^2A(z_n)\big),z_n\big)=-2\beta\int_D|A(y)|^2A(z_n):\nabla z_ndx+\int_{\partial D}|A(y)|^2(A(z_n)\eta)\cdot z_ndS
\vspace{1mm}\\
&=-\beta\int_D|A(y)|^2|A(z_n)|^2dx \leq 0,
\end{align*}
where we used the Navier boundary conditions to  cancel the boundary term. Concerning $I_1^4$, we have
\begin{align*}
I_1^4= 4\beta\left(\text{ div} \left[\left(A(z_n): A(y)\right)A(y)\right],z_n\right)=4\beta\int_D\text{div}\big [\left(A(z_n): A(y)\right)A(y)\big]\cdot z_ndx.
\end{align*}
A similar reasoning yields
\begin{align*}
I_1^4= -4\beta\int_D[(A(z_n):A(y))A(y)\big]: \nabla z_ndx+4\beta\int_{\partial D}\left(A(z_n):A(y)\right)(A(y)\eta)\cdot z_n dS.
\end{align*}
Again, using the Navier boundary conditions, we  deduce
\begin{align*}
\vert I_1^4\vert&= \vert 4\beta\int_D[\left(A(z_n):A(y)\right)A(y)\big]: \nabla z_ndx\vert\leq C_3  \Vert z_n\Vert_{H^1}^2\Vert  y\Vert_{W^{1,\infty}}^2\leq C_3  \Vert z_n\Vert_{V}^2\Vert  y\Vert_{\widetilde{W}}^2.
\end{align*}
Gathering the previous estimates, there exists $C>0$ independent of $n$ such that
\begin{align*}
\vert I_1^1+I_1^2+I_1^3+I_1^4\vert \leq  C \Vert z_n\Vert_{V}^2\Vert  y\Vert_{\widetilde{W}}(1+\Vert  y\Vert_{\widetilde{W}}).
\end{align*}
For any $t\in ]0,T]$, we integrate over the time variable to obtain
\begin{align*}
\big(\Vert z_n(t)\Vert_2^2+2\alpha_1\Vert  \mathbb{D}z_n(t)\Vert_2^2\big)
+4\nu\int_0^t\Vert \mathbb{D} z_n\Vert_2^2ds&\leq C\int_0^t   \Vert z_n\Vert_{V}^2\Vert  y\Vert_{\widetilde{W}}(1+\Vert  y\Vert_{\widetilde{W}}) ds+C\int_0^{t}\Vert \psi(s)\Vert_{2}^2ds.
\end{align*}
Thanks to Lemma \ref{Lemma-H3}, for any $t\in ]0,T]$, we have
\begin{align*}
&\sup_{r\in [0, t]}\Vert z_n(r)\Vert_V^2+4\nu\int_0^t\Vert \mathbb{D} z_n(s)\Vert_2^2ds\leq  C \sup_{r\in [0, t]}[M_4(r)+M_4^2(r)]\int_0^t  \Vert z_n(s)\Vert_{V}^2 ds+C\int_0^{t}\Vert \psi(s)\Vert_{2}^2ds.
\end{align*}
Now, the Gronwall's inequality gives that  $(z_n)_n$  is bounded in $L^\infty(0,T;V)$, i.e.
\begin{equation}
\label{Vest-Zn}
\sup_{r\in [0, t]}\Vert z_n(r)\Vert_V^2+4\nu \int_0^{ t}\Vert 
\mathbb{D} z_n(s)\Vert_2^2ds\leq C(T)\int_0^{t}\Vert \psi(s)\Vert_{2}^2ds.
\end{equation}
\subsubsection{Estimate in the space $W$ for $z_n$}
Let $\tilde{f}_n$ be the solutions of \eqref{Stokes} for $f=f(z_n)$. Then 
\begin{align}
(\tilde{f}_n,h_i)_V=(f(z_n),h_i),\quad \text{  for each } i.
\end{align}
Multiplying \eqref{approximationz-n} by $\mu_i$ and using \eqref{basis1}, we get
\begin{align}
(\partial_tz_n,h_i)_W=(\tilde{f}_n,h_i)_W,
\end{align}
Multiplying these equalities by $c_i(t)$ and summing from $i=1$ to $n$, it follows that
\begin{equation*}
\begin{array}{ll}
(\partial_tz_n,z_n)_V+(\mathbb{P} v(\partial_tz_n),\mathbb{P} v(z_n))&= (\partial_tz_n,z_n)_W=(\tilde{f}_n,z_n)_V+(\mathbb{P} v\tilde{f}_n,\mathbb{P} v(z_n))
\vspace{1mm}\\
&=(f(z_n),z_n)+(f(z_n),\mathbb{P} v(z_n)).
\end{array}
\end{equation*}
By using \eqref{estimateV}, the last equality reduces to 
\begin{equation*}
\begin{array}{ll}
\dfrac{d}{dt}\Vert\mathbb{P} v(z_n)\Vert_{2}^2=2(\partial_t\mathbb{P} v(z_n),\mathbb{P} v(z_n))=2(f(z_n),\mathbb{P} v(z_n)).
\end{array}
\end{equation*}
Let us estimate $2( f(z_n),\mathbb{P}v(z_n))$. We write
\begin{equation*}
\begin{array}{ll}
2( f(z_n),\mathbb{P}v(z_n))&=2\big(\psi-\nabla \pi_n+\nu \Delta z_n, \mathbb{P}v(z_n)\big)+2\beta\big(\text{ div}\big [ |A(y)|^2A(z_n)\big],\mathbb{P}v(z_n)\big)\vspace{1.5mm}\\
& -2\big((y\cdot \nabla )v(z_n)+(z_n\cdot \nabla)v(y),\mathbb{P}v(z_n)\big)\vspace{1.5mm}\\
&-2\big(\sum_{j}[v(z_n)^j\nabla y^j+v(y)^j\nabla z_n^j],\mathbb{P}v(z_n)\big)\vspace{1.5mm}\\
& +2(\alpha_1+\alpha_2)\big(\text{div}[A(y)A(z_n)+A(z_n)A(y)],\mathbb{P}v(z_n)\big) \vspace{1.5mm} \\
&+4\beta\big(\text{ div} \big [\left(A(z_n)\cdot A(y)\right)A(y)\big],\mathbb{P}v(z_n)\big)=J_1^0+J_1^1+J_1^2+J_1^3+J_1^4.
\end{array}
\end{equation*}
The first term verifies
\begin{align*}
J_1^0=2\big(\psi-\nabla \pi_n+\nu \Delta z_n, \mathbb{P}v(z_n)\big)\leq 2\vert (\psi,\mathbb{P}v(z_n))\vert+C\Vert z_n\Vert_W^2. 
\end{align*}
Taking into account that $b(y,\mathbb{P}v(z_n),\mathbb{P}v(z_n))=0$, we have
\begin{align*}
J_1^2&=-2\big(b(y,v(z_n)-\mathbb{P}v(z_n),\mathbb{P}v(z_n))+b(z_n,v(y),\mathbb{P}v(z_n))\vspace{1mm}\\&\qquad+b(\mathbb{P}v(z_n),y,v(z_n))+b(\mathbb{P}v(z_n),z_n,v(y))\big).
\end{align*}
Thanks to the H\"older inequality, we deduce
\begin{align*}
\vert J_1^2\vert 
&\leq C\Vert y\Vert_{\infty}\Vert v(z_n)-\mathbb{P}v(z_n)\Vert_{H^1}\Vert\mathbb{P}v(z_n)\Vert_{2}+C\Vert z_n\Vert_{\infty}\Vert v(y)\Vert_{H^1}\Vert\mathbb{P}v(z_n)\Vert_{2}\vspace{1mm}\\
&\quad+C\Vert \nabla y\Vert_{\infty}\Vert v(z_n)\Vert_{2}\Vert\mathbb{P}v(z_n)\Vert_{2}+C\Vert \nabla z_n\Vert_{4}\Vert \mathbb{P}v(z_n)\Vert_{2}\Vert v(y)\Vert_{4}.
\end{align*}
We recall that $\widetilde{W} \hookrightarrow W^{1,\infty}(D)\cap  W^{2,4}(D)$, $W \hookrightarrow W^{1,4}(D)\cap  L^\infty(D)$. Hence 
\begin{equation*}
\vert J_1^2\vert\leq K_1\Vert y\Vert_{\widetilde{W}}\Vert z_n\Vert_{W}^2,
\end{equation*}
where we used $\Vert v(z_n)-\mathbb{P}v(z_n)\Vert_{H^1} \leq K_1\Vert z_n\Vert_{H^2}$ (see \cite[Lem. 5]{Bus-Ift-2}). \\

Concerning $J_1^3$, we have
\begin{equation*}
J_1^3
=2(\alpha_1+\alpha_2)\int_D \text{div}[A(y)A(z_n)+A(z_n)A(y)]\cdot\mathbb{P}v(z_n)dx. 
\end{equation*}
We observe that $\text{div}[A(y)A(z_n)+A(z_n)A(y)]$ can be expressed as the sum of terms of one of the following forms $ \mathcal{D}(y)\mathcal{D}^2(z_n)$ or $ \mathcal{D}(z_n)\mathcal{D}^2(y)$ \footnote{ $\mathcal{D}^k(u)=(u,\nabla u,\cdots,\nabla^ku)$ is the vector of $\mathbb{R}^{d^{k+1}+\cdots+d^2+d}$ whose components are the components  of $u$ together with the derivatives of order up to $k$ of these components.}. Therefore,
\begin{align*}
\vert \text{div}[A(y)A(z_n)+A(z_n)A(y)]\vert \leq  C(\vert \nabla y\vert \vert \nabla^2z_n\vert+\vert \nabla z_n\vert \vert \nabla^2y\vert).
\end{align*}
Hence
$\vert J_1^3\vert\leq C\int_D(\vert \nabla y\vert \vert \nabla^2z_n\vert+\vert \nabla z_n\vert \vert \nabla^2y\vert)\cdot\vert \mathbb{P}v(z_n)\vert dx.$
The H\"older inequality yields
\begin{align*} 
\vert J_1^3\vert&\leq C\Vert \nabla y\Vert_{\infty} \Vert \nabla^2z_n\Vert_2 \Vert \mathbb{P}v(z_n) \Vert_2+ C\Vert \nabla^ 2 y\Vert_{4} \Vert \nabla z_n\Vert_4 \Vert \mathbb{P}v(z_n) \Vert_2\\
&\leq K_2\Vert y\Vert_{\widetilde{W}}\Vert z_n\Vert_{W}^2.
\end{align*}
On the other hand, we recall that
\begin{align*}
J_1^1+J_1^4&=2\beta\big(\text{ div}\big [ |A(y)|^2A(z_n)\big],\mathbb{P}v(z_n)\big)+4\beta\big(\text{ div} \big [A(z_n): A(y)A(y)\big],\mathbb{P}v(z_n)\big)\vspace{1mm}\\
&=2\beta\int_D\text{ div}\big [ |A(y)|^2A(z_n)\big]\cdot\mathbb{P}v(z_n)dx
+4\beta\int_D\text{ div} \big [(A(z_n):A(y))A(y)\big]\cdot\mathbb{P}v(z_n)dx,
\end{align*}
and notice that  $\text{ div}\big [ |A(y)|^2A(z_n)\big]$ and $\text{ div} \big [(A(z_n): A(y))A(y)\big]$ can be expressed as the sum of terms of one of the following forms $ \mathcal{D}(y)\mathcal{D}^2(z_n)\mathcal{D}(y)$ or $ \mathcal{D}(z_n)\mathcal{D}^2(y)\mathcal{D}(y)$.
Similar arguments as  in the estimation of $J_1^2$ give
$\vert J_1^3+J_1^4\vert \leq  K_3\Vert y\Vert_{\widetilde{W}}^2\Vert z_n\Vert_{W}^2.$
Thus there exists $K>0$ such that
\begin{align}\label{estimateWzn}
\vert J_1^1+J_1^2+J_1^3+J_1^4\vert \leq  K\Vert y\Vert_{\widetilde{W}}(1+\Vert y\Vert_{\widetilde{W}})\Vert z_n\Vert_{W}^2.
\end{align}
Therefore, for any $t\in [0,T]$,
\begin{align*}
2\int_0^t\vert (f(z_n),\mathbb{P}v(z_n))\vert ds&\leq C \sup_{r\in [0, t]}[M_4(r)+M_4^2(r)]\int_0^t\Vert z_n\Vert_{W}^2 ds+2\int_0^t\vert (\psi,\mathbb{P}v(z_n))\vert ds+C\int_0^t\Vert z_n\Vert_W^2 ds. 
\end{align*}
Thanks to Lemma \ref{Lemma-H3}, we obtain
\begin{align*}
2\int_0^t\vert (f(z_n),\mathbb{P}v(z_n))\vert ds\leq C\int_0^t\Vert \psi(s)\Vert_{2}^2ds +C\int_0^t\Vert z_n\Vert_W^2 ds 
\end{align*}
and 
\begin{align*}
\sup_{r\in [0, t]}\Vert\mathbb{P} v(z_n(r))\Vert_{2}^2\leq 2\int_0^t\vert (f(z_n),\mathbb{P}v(z_n))\vert ds\leq C\int_0^t\Vert \psi(s)\Vert_{2}^2ds +C\int_0^t\Vert z_n(s)\Vert_W^2 ds. 
\end{align*}
Gronwall's inequality ensures that there exists $C >0$ such that
\begin{align}
\sup_{r\in [0, t]}\Vert\mathbb{P} v(z_n(r))\Vert_{2}^2\leq  C(T)\int_0^t\Vert \psi(s)\Vert_{2}^2ds.  
\end{align}
Gathering \eqref{Vest-Zn} and the last inequality, we write
\begin{align*}
\sup_{r\in [0, t]}\Vert\mathbb{P} v(z_n(r))\Vert_{2}^2+\sup_{r\in [0, t]}\Vert z_n(r)\Vert_V^2+4\nu \int_0^{ t}\Vert \mathbb{D} z_n(s)\Vert_2^2ds\leq C(T)\int_0^{t}\Vert \psi(s)\Vert_{2}^2ds.
\end{align*}
We conclude that there exists $M_1>0$ such that 
\begin{align}\label{Wznestimate}
\sup_{r\in [0, t]}\Vert z_n(r)\Vert_W^2 \leq M_1(T)\int_0^T\Vert \psi(s)\Vert_2^2ds \quad \text{  for  any } t\in [0,T].
\end{align}
\subsubsection{Estimate in the space $V$ for  $\partial_t z_n$}   Multiplying \eqref{approximationz-n} by $\partial_tc_i(t)$ and summing from  $i=1$ to $n$, we get
\begin{align*}
\begin{cases}
(\partial_tv(z_n),\partial_tz_n)&=\Big(\psi+\nu \Delta z_n-(y\cdot \nabla )v(z_n)-(z_n\cdot \nabla)v(y)-\sum_{j}v(z_n)^j\nabla y^j\\[1mm]&\quad -\sum_{j}v(y)^j\nabla z_n^j
+(\alpha_1+\alpha_2)\text{div}[A(y)A(z_n)+A(z_n)A(y)]\\[1mm]& \quad+\beta\text{ div}\big [ |A(y)|^2A(z_n)\big] +2\beta\text{ div} \big [\left(A(z_n): A(y)\right)A(y)\big],\partial_tz_n\Big)\\
&:=(f(z_n),\partial_tz_n).
\end{cases}
\end{align*}
Notice that
\begin{align*}
2(\partial_t v(z_n),\partial_tz_n)&=2(\partial_t (z_n),\partial_tz_n)+4\alpha_1(\partial_t \mathbb{D}z_n,\mathbb{D}\partial_tz_n)=2\Vert \partial_tz_n\Vert_2^2+4\alpha_1\Vert  \mathbb{D}\partial_tz_n\Vert_2^2.
\end{align*}
Let us estimate $(f(z_n),\partial_tz_n)$. 
\begin{align*}
(f(z_n),\partial_tz_n)&=\Big(-(y\cdot \nabla )v(z_n)-(z_n\cdot \nabla)v(y)-\sum_{j}(v(z_n)^j\nabla y^j-v(y)^j\nabla z_n^j),\partial_tz_n\Big)\vspace{1mm}\\
&+\Big((\alpha_1+\alpha_2)\text{div}[A(y)A(z_n)+A(z_n)A(y)] +\beta\text{ div}\big [ |A(y)|^2A(z_n)\big],\partial_tz_n\Big)\vspace{1mm}\\
&+2\beta\Big(\text{ div} 
\big [\left(A(z_n):A(y)\right)A(y)\big],\partial_tz_n\Big)+\Big(\psi+\nu \Delta z_n,\partial_tz_n\Big)\vspace{1mm}\\
&\leq (\Vert \psi\Vert_{2}+\nu \Vert z_n\Vert_{W})\Vert \partial_t z_n\Vert_{2}+\vert b(y,v(z_n),\partial_t z_n)\vert+\vert b(z_n,v(y),\partial_t z_n)\vert
\vspace{2mm}\\&+\vert b(\partial_t z_n,z_n,v(y)) \vert+C(\alpha_1+\alpha_2)\int_D(\vert \nabla y\vert \vert \nabla^2z_n\vert+\vert \nabla z_n\vert \vert \nabla^2y\vert)\cdot\vert \partial_tz_n\vert dx\vspace{2mm}\\
&+\vert b(\partial_t z_n,y,v(z_n)) \vert+C\beta \int_D(\vert \nabla y\vert^2 \vert \nabla^2z_n\vert+\vert \nabla z_n\vert \vert \nabla^2y\vert\vert \nabla y\vert )\cdot\vert \partial_tz_n\vert dx.
\end{align*}
A similar argument to the one used to get  \eqref{Vest-Zn} yields
\begin{equation*}
\begin{array}{ll}
\vert b(y,v(z_n),\partial_t z_n)\vert&+\vert b(z_n,v(y),\partial_t z_n)\vert+\vert b(\partial_t z_n,y,v(z_n)) \vert+\vert b(\partial_t z_n,z_n,v(y)) \vert
\vspace{1mm}\\&\leq C\Vert y(t)\Vert_{H^{3}}\Vert z_n(t)\Vert_{W}\big( \Vert \partial_tz_n(t)\Vert_2+\Vert  \mathbb{D}\partial_tz_n(t)\Vert_2 \big).
\end{array}
\end{equation*}
On the other hand, we have
\begin{align*}
&\int_D(\vert \nabla y\vert \vert \nabla^2z_n\vert
+\vert \nabla z_n\vert \vert \nabla^2y\vert)\cdot\vert \partial_tz_n\vert dx \leq C\Vert y\Vert_{W^{1,\infty}}\Vert z_n\Vert_{W}\Vert \partial_t z_n\Vert_2\vspace{1mm}\\&
\qquad \qquad\qquad\qquad+C\Vert z_n\Vert_{W^{1,4}}\Vert y\Vert_{W^{2,4}}\Vert \partial_t z_n\Vert_2\leq C\Vert y(t)\Vert_{H^{3}}\Vert z_n(t)\Vert_{W}\Vert \partial_t z_n(t)\Vert_2,
\end{align*}
and 
\begin{align*}
&\int_D(\vert \nabla y\vert^2 \vert \nabla^2z_n\vert+\vert \nabla z_n\vert \vert \nabla^2y\vert\vert \nabla y\vert )\cdot\vert \partial_tz_n\vert dx\vspace{1mm}\\
&\qquad\leq C\Vert y\Vert_{W^{1,\infty}}^2\Vert z_n\Vert_{W}\Vert \partial_t z_n\Vert_2+C\Vert y\Vert_{W^{1,\infty}}\Vert z_n\Vert_{W^{1,4}}\Vert y\Vert_{W^{2,4}}\Vert \partial_t z_n\Vert_2\vspace{1mm}\\
&\qquad\leq C\Vert y(t)\Vert_{H^{3}}^2\Vert z_n(t)\Vert_{W}\Vert \partial_t z_n(t)\Vert_2.
\end{align*}
Therefore
\begin{align*}
&\Vert \partial_tz_n\Vert_2^2+2\alpha_1\Vert  \mathbb{D}\partial_tz_n\Vert_2^2\vspace{5mm}\\
&\quad\leq C (\Vert \psi\Vert_{2}+\nu \Vert z_n\Vert_{W})\Vert \partial_t z_n\Vert_{2}+C\Vert y(t)\Vert_{H^{3}}\Vert z_n(t)\Vert_{W}\big( \Vert \partial_tz_n(t)\Vert_2+\Vert  \mathbb{D}\partial_tz_n(t)\Vert_2 \big)\vspace{5mm}\\
&\qquad+C\Vert y(t)\Vert_{H^{3}}\Vert z_n(t)\Vert_{W}\Vert \partial_t z_n(t)\Vert_2+C\Vert y(t)\Vert_{H^{3}}^2\Vert z_n(t)\Vert_{W}\Vert \partial_t z_n(t)\Vert_2.
\end{align*}
For any $\delta>0$, the Young inequality ensures that
\begin{align*}
&\Vert \partial_tz_n\Vert_2^2+2\alpha_1\Vert  \mathbb{D}\partial_tz_n\Vert_2^2\\[0.1cm]
&\leq  C(\Vert \psi\Vert_{2}^2+\nu^2\Vert z_n\Vert_{W}^2)+2\delta\Vert \partial_t z_n\Vert_{2}^2+C\Vert y(t)\Vert_{H^{3}}^2\Vert z_n(t)\Vert_{W}^2+\delta\Vert  \mathbb{D}\partial_tz_n(t)\Vert_2^2\\[0.1cm]
&\quad+C\Vert y(t)\Vert_{H^{3}}^2\Vert z_n(t)\Vert_{W}^2+\delta\Vert \partial_t z_n(t)\Vert_2^2+C\Vert y(t)\Vert_{H^{3}}^4\Vert z_n(t)\Vert_{W}^2+\delta\Vert \partial_t z_n(t)\Vert_2^2.
\end{align*}
An appropriate choice of $\delta$ and integration with respect to time  $t$ give
\begin{align*}
&\int_0^T(\Vert \partial_tz_n\Vert_2^2+\alpha_1\Vert  \mathbb{D}\partial_tz_n\Vert_2^2)ds\leq  C\int_0^T\Vert \psi(s)\Vert_{2}^2ds+C\int_0^T(1+\Vert y(s)\Vert_{H^{3}}^2+\Vert y(s)\Vert_{H^{3}}^4)\Vert z_n(s)\Vert_{W}^2ds\\
&\quad\leq C\int_0^T\Vert \psi(s)\Vert_{2}^2ds+C\sup_{r\in [0, T]}[1+M_4^2(r)+M_4^4(r)]\int_0^T\Vert z_n(s)\Vert_{W}^2ds:= C_1(T).
\end{align*}
\begin{remark}\label{rmq-global-time}
	Before passing to the limit in the approximate problem, let us note  that \eqref{Wznestimate} guarantees  the global in time  existence of $z_n$, \textit{i.e.}  $z_n\in \mathcal{C}([0,T],W_n).$
\end{remark}
Finally, we derive the following lemma.
\begin{lemma}\label{Lem-esti-zn}
	Assume $\psi$ satisfies \eqref{psi}. Then, there exists a unique  solution $z_n\in \mathcal{C}([0,T],W_n)$ to \eqref{approximationz-n}. Moreover, there exist  positive constants $C(T),C_1(T)$, which are independent on the index $n$, such that the following estimates hold for each $t\in [0,T]$
	\begin{align*}
	\sup_{r\in [0,t]}\Vert z_n(r)\Vert_V^2+4\nu \int_0^{ t}\Vert 
	\mathbb{D} z_n(s)\Vert_2^2ds&\leq C(T)\int_0^{T}\Vert \psi(s)\Vert_{2}^2ds,\vspace{2mm}\\
	\sup_{r\in [0, t]}\Vert z_n(r)\Vert_W^2 &\leq C(T)\int_0^T\Vert \psi(s)\Vert_2^2ds,\vspace{2mm}\\
	\int_0^T(\Vert \partial_tz_n\Vert_2^2+\alpha_1\Vert  \mathbb{D}\partial_tz_n\Vert_2^2)ds&\leq  C_1(T).
	\end{align*}
\end{lemma}
\subsection{Transition to the limit}
\begin{prop}\label{exis-linearize}
	Let $\psi$ satisfies \eqref{psi}, then there exists a unique solution $z$ to \eqref{Linearized} in the sense of Definition \ref{Def-lin}
	and	satisfying the following estimates
	\begin{align*}
	\sup_{r\in [0,t]}\Vert z(r)\Vert_V^2+4\nu \int_0^{ t}\Vert  z(s)\Vert_V^2ds&\leq C(T)\int_0^{T}\Vert \psi(s)\Vert_{2}^2ds\quad \text{  for  any } t\in ]0,T],\vspace{2mm}\\
	\sup_{r\in [0, t]}\Vert z(r)\Vert_W^2 &\leq C(T)\int_0^T\Vert \psi(s)\Vert_2^2ds,\vspace{2mm}\\
	\int_0^T(\Vert \partial_tz\Vert_2^2+\alpha_1\Vert  
	\mathbb{D}\partial_tz\Vert_2^2)ds&\leq  C_1(T).	
	\end{align*}
\end{prop}
\begin{proof}
	By compactness with respect to the weak-$*$ topology in the spaces $L^\infty(0,T;V)$, $L^\infty(0,T;W)$ and  the weak topology in the space $L^2(0,T;V)$  there exists $z \in L^\infty(0,T;W)$ such that the following convergences hold, up to sub-sequences (denoted  in the same way as the sequences)
	\begin{equation}
	\label{weak-cv-zn}
	\begin{array}{ll}
	z_n &\overset{\ast}{\rightharpoonup} z \quad \text{  in }  L^\infty(0,T;V),\vspace{2mm}\\
	z_n &\overset{\ast}{\rightharpoonup} z \quad \text{  in }  L^\infty(0,T;W),\vspace{2mm}\\
	\partial_tz_n &\rightharpoonup \partial_tz \quad \text{  in }  L^2(0,T;V).
	\end{array}
	\end{equation}
	From \eqref{weak-cv-zn}, we  deduce  that $z\in \mathcal{C}([0,T],V)$ and therefore $z_n(0)=0$ converges to $z(0)$ in $V$, i.e. $z(0)=0$.
	We recall that $z_n$ solves the equation
	\begin{equation*}
	\begin{array}{ll}
	(\partial_tv(z_n),\phi)&= \Big(\psi+\nu \Delta z_n-(y\cdot \nabla )v(z_n)-(z_n\cdot \nabla)v(y)-\sum_{j}v(z_n)^j\nabla y^j\vspace{2mm}\\
	&-\sum_{j}v(y)^j\nabla z_n^j
	+(\alpha_1+\alpha_2)\text{div}[A(y)A(z_n)+A(z_n)A(y)]\vspace{2mm}\\ &+\beta\text{ div}\big [ |A(y)|^2A(z_n)\big]+2\beta\text{ div} \big [(A(z_n): A(y))A(y)\big],\phi\Big), \quad\forall \phi \in W_n.  
	\end{array}
	\end{equation*}
	Setting $\phi=h_i$ and passing to the limit, as $n\to \infty$, we deduce
	\begin{align*}
	&(\partial_tv(z),h_i)
	+2\nu(\mathbb{D} z,\mathbb{D}h_i)+b(y,v(z),h_i)+b(z,v(y),h_i)+b(h_i,y,v(z))\\[2mm]&+b(h_i,z,v(y))+(\alpha_1+\alpha_2)\big(A(y)A(z)+A(z)A(y),\nabla h_i\big)  +\beta\big( |A(y)|^2A(z), \nabla h_i\big)
	\\[2mm]&\quad+2\beta \big((A(z):
	A(y))A(y),\nabla h_i\big)= (\psi,h_i)
	\end{align*}
	for each $ i\in \mathbb{N}.$ Finally, a standard density argument gives the claimed result, namely  $z\in L^2(0,T;W)$ is a solution of \eqref{Linearized}
	in the sense of Definition \ref{Def-lin}.
\end{proof}
\section{A Stability result } 
\label{S5}
The main task of this  section is to  establish  a stability result for the solution of the state equation. This is a crucial  step to study  the G\^ateaux derivative of the control-to-state mapping.
\vspace{2mm}\\
Let us recall  the relation  (see \cite[Appendix]{Bus-Ift-1})
\begin{equation}
\label{3rd-equivalent}
\begin{array}{ll}
\dfrac{1}{2}\nabla(\alpha_1\vert \nabla y\vert^2&-\vert y\vert^2)-(y\cdot \nabla)y+\text{ div}(N(y))
\vspace{2mm}\\&= -(y\cdot \nabla)v(y)-\sum_{j}v(y)^j\nabla y^j+(\alpha_1+\alpha_2)\text{ div}(A(y)^2),
\end{array}
\end{equation}
which allows to write  the equations \eqref{I} in the following form
\begin{equation}
\label{I-equi}
\left\{\begin{array}{ll}
\partial_t(v(y))=-\nabla \mathbf{\bar P}+\nu \Delta y-(y\cdot \nabla)y+\text{ div}(N(y))+\text{ div}(S(y))+U \quad &\hspace*{-0.5cm}\text{in } D\times (0,T),\vspace{2mm}\\
\text{div}(y)=0 \quad &\hspace*{-0.5cm}\text{in } D\times  (0,T),\vspace{2mm}\\
y\cdot \eta=0, \quad [\eta \cdot \mathbb{D}(y)]\cdot \tau=0  \quad &\hspace*{-0.5cm}\text{on } \partial D \times (0,T),\vspace{2mm}\\
y(x,0)=y_0(x) \quad &\hspace*{-0.5cm}\text{in } D,
\end{array}
\right.
\end{equation}
where
\begin{equation*}
\begin{array}{ll}
S(y)&:=\beta \Big ( |A(y)|^2A(y)\Big),\vspace{1mm}\\
N(y)&:=\alpha_1\big(  y \cdot \nabla A(y)+(\nabla y)^TA(y)+A(y)\nabla y\big)+\alpha_2(A(y))^2.
\end{array}
\end{equation*}

\begin{itemize}
	\item[$H_4$:] Assume that 
	$
	U_1,\,U_2 \in L^2(0,T;(H^1(D))^2); \quad 
	y_0^1,\,y_0^2 \in \widetilde{W}.
	$
\end{itemize}
\begin{theorem}
	Let us take $U_1,\,U_2 $ and $y_0^1,\,y_0^2$ verifying $H_4$, and consider the  corresponding solutions of \eqref{I}
	\begin{equation*}
	\begin{array}{ll}
	y_1,y_2 \in L^\infty(0,T;\widetilde{W}).
	\end{array}
	\end{equation*}
	Then, there exists a positive constant $\widetilde C$, which depends only on the data such that 
	the following estimate holds
	\begin{equation}\label{stability-estimate}
		\sup_{r\in [0, T]}\Vert y_1(r)-y_2(r)\Vert_{W}^2\leq \widetilde C\big[ \Vert y_0^1-y_0^2\Vert_{W}^2+\int_0^T\Vert U_1(s)-U_2(s)\Vert_{2}^2ds\big].
	\end{equation}
\end{theorem}
\begin{proof}
	Let $y_1$ and $y_2$ be two solutions of \eqref{I} associated  with the external forces $U_1$ and $U_2$   and the  initial data $y_0^1$ and $y_0^2$, respectively.\vspace{2mm}\\
	Denoting  $y=y_1-y_2$ and $y_0=y_0^1-y_0^2$,  we can verify that  $y$ solves the system
	\begin{equation}
	\label{Ito1}
	\left\{\begin{array}{ll}
	\partial_t(v(y))=-\nabla(\mathbf{\bar P}_1-\mathbf{\bar P}_2)+\nu \Delta y-\big[(y\cdot \nabla)y_1+(y_2\cdot \nabla)y\big]&\vspace{2mm}\\\qquad+\text{ div}(N(y_1)-N(y_2))
	+\text{ div}(S(y_1)-S(y_2))+(U_1-U_2) \quad &\text{in } D\times  (0,T),\vspace{2mm}\\
	\text{div}(y)=0 \quad &\text{in } D\times  (0,T),
	\vspace{1mm}\\
	y\cdot \eta=0, \quad [\eta \cdot \mathbb{D}(y)]\cdot \tau=0  \quad &\text{on } \partial D\times (0,T),\vspace{1mm}\\
	y(x,0)=y_0(x) \quad & \text{in } D.
	\end{array}
	\right.
	\end{equation}
	Let us test \eqref{Ito1}$_ 1$  by $y$. Then we have
	\begin{align*}
	&	\partial_t(\Vert y\Vert_V^2)+4\nu\Vert \mathbb{D} y\Vert_2^2\vspace{2mm}\\&\quad=-2\int_D\big[(y\cdot \nabla)y_1+(y_2\cdot \nabla)y\big]y dx
	+2\langle \text{ div}(N(y_1)-N(y_2)), y\rangle \vspace{2mm}\\&\qquad+
	2\langle \text{ div}(S(y_1)-S(y_2)),y\rangle+2\int_D(U_1-U_2)\cdot ydx=I_1+I_2+I_3+I_4.
	\end{align*}
	We will estimate $I_i, i=1,\cdots,4$. Since  $V\hookrightarrow L^4(D)$, the first term verifies
	\begin{align*}
	\vert I_1\vert 
	&=2\left\vert \int_D(y\cdot \nabla)y_1\cdot y dx\right\vert \leq C\Vert y\Vert_4^2\Vert \nabla y_1\Vert_2\leq C\Vert y\Vert_V^2\Vert \nabla y_1\Vert_2 \leq  C\Vert y\Vert_V^2\Vert y_1\Vert_{H^3}. 
	\end{align*}
	After an integration by parts, the term $I_3$, can be treated 
	using the same arguments as in \cite[Sect.3]{Bus-Ift-2}, the term on the boundary vanishes and  we have
	\begin{align*}
	I_3&=	2\langle \text{ div}(S(y_1)-S(y_2)),y_1-y_2\rangle=-2\int_D(S(y_1)-S(y_2)): \nabla ydx
	\vspace{2mm}\\&=-\dfrac{\beta}{2}(\int_D(|A(y_1)|^2-|A(y_2)|^2)^2dx+\int_D(|A(y_1)|^2+|A(y_2)|^2)|A(y_1-y_2)|^2dx) \leq 0.
	\end{align*}
	Concerning $I_4$, one has 
	\begin{equation*}
	|I_4|=	2\left|\int_D(U_1-U_2)\cdot ydx\right| \leq \Vert U_1-U_2\Vert_2^2+\Vert y\Vert_2^2\leq \Vert U_1-U_2\Vert_2^2+\Vert y\Vert_V^2.
	\end{equation*}
	Let us estimate the  term $I_2$. Integrating by 
	parts and taking into account that the boundary terms vanish (see \cite[Sect. 3]{Bus-Ift-2}), we deduce
	\begin{align*}
	I_2&= 2\langle \text{ div}(N(y_1)-N(y_2)), y\rangle =-2 \int_D(N(y_1)-N(y_2)): \nabla ydx\vspace{2mm}\\
	&=-\alpha_2\int_D\big(A(y_1)^2-A(y_2)^2\big): A(y)dx-\alpha_1\int_D\big(y_1 \cdot \nabla A(y_1)-y_2 \cdot \nabla A(y_2)\big): A(y)dx\vspace{2mm}\\
	& -\alpha_1\int_D((\nabla y_1)^TA(y_1)+A(y_1)\nabla y_1-(\nabla y_2)^TA(y_2)-A(y_2)\nabla y_2): A(y)dx\vspace{2mm}\\&=-\alpha_2I_2^1-\alpha_1I_2^2-\alpha_1I_2^3. 	
	\end{align*}
	Since
	\begin{align*}
	I_2^1&=	\int_D\big(A(y_1)^2-A(y_2)^2\big): A(y)dx=\int_D\big(A(y)A(y_1)+A(y_2)A(y)\big): A(y)dx;\\[1mm]
	I_2^2&=\int_D\big(y_1 \cdot \nabla A(y_1)-y_2 \cdot \nabla A(y_2)\big): A(y)dx\\&=\int_D\big(y_1 \cdot \nabla A(y_1-y_2)+(y_1-y_2)\cdot \nabla A(y_2)\big): A(y)dx=\int_D\big(y\cdot \nabla A(y_2)\big): A(y)dx;\\[1mm]
	I_2^3&=\int_D((\nabla y_1)^TA(y_1)+A(y_1)\nabla y_1-(\nabla y_2)^TA(y_2)-A(y_2)\nabla y_2): A(y)dx\\&=2\int_D\big(A(y_1)A(y)):\nabla y_1-(A(y_2)A(y)):\nabla y_2\big)dx\\&=2\int_D\big((A(y))^2:\nabla y_1+(A(y_2)A(y)):\nabla y\big) dx;
	\end{align*}
	the H\"older's inequality and the embedding $H^1(D) \hookrightarrow L^4(D)$ yield
	\begin{align*}
	\vert I_2^1\vert &\leq \int_D\vert \big(A(y)A(y_1)+A(y_2)A(y)\big)\vert : \vert A(y)\vert dx \leq C(\Vert y_1\Vert_{W^{1,\infty}}+\Vert y_2\Vert_{W^{1,\infty}})\Vert \nabla y\Vert_{2}^2;\vspace{2mm}\\
	\vert I_2^2\vert &\leq \int_D\vert \big(y\cdot \nabla A(y_2)\big): A(y) \vert  dx\leq C\Vert y\Vert_4\Vert y_2\Vert_{W^{2,4}}\Vert \nabla y\Vert_{2}\leq  C\Vert y_2\Vert_{W^{2,4}}\Vert \nabla y\Vert_{2}^2;\vspace{2mm}\\
	\vert I_2^3\vert &\leq C\int_D\vert \big((A(y))^2:\nabla y_1+(A(y_2)A(y)):\nabla y\big)\vert dx  \leq C(\Vert y_1\Vert_{W^{1,\infty}}+\Vert y_2\Vert_{W^{1,\infty}})\Vert \nabla y\Vert_{2}^2.
	\end{align*}
	Then the embedding  $H^3(D)\hookrightarrow W^{2,4}(D)\cap W^{1,\infty}(D)$ gives
	$\vert I_2\vert \leq C (\Vert y_1\Vert_{H^3}+\Vert y_2\Vert_{H^3})\Vert  y\Vert_{V}^2.$
	By gathering the previous estimates, we obtain
	\begin{align}
	\label{stabilityVnorm}
	\Vert y(t)\Vert_V^2&+4\nu \int_0^t\Vert \mathbb{D}y\Vert_{2}^2ds\vspace{2mm} \nonumber\\&\leq \Vert y_0\Vert_{V}^2+M_0\int_0^t(\Vert y_1\Vert_{H^3}+\Vert y_2\Vert_{H^3}+1)\Vert  y\Vert_{V}^2ds+\int_0^t\Vert U_1-U_2\Vert_{2}^2ds.
	\end{align}
	Now,   multiplying \eqref{Ito1}$_ 1$  by $\mathbb{P} v (y)$, we write
	\begin{align*}
	\partial_t\Vert \mathbb{P} v (y)\Vert_2^2&=2\langle \nu \Delta y,\mathbb{P}v(y)\rangle-2\langle \big[(y\cdot \nabla)y_1+(y_2\cdot \nabla)y\big],\mathbb{P}v(y) \rangle+2\langle(U_1-U_2), \mathbb{P}v(y)\rangle\vspace{2mm}\\
	&\quad+2\langle \text{ div}(N(y_1)-N(y_2)), \mathbb{P}v(y)\rangle+
	2\langle \text{ div}(S(y_1)-S(y_2)),\mathbb{P}v(y)\rangle.
	\end{align*}
	Using \eqref{3rd-equivalent} and knowing that    $\mathbb{P}v(y)$ is 
	divergence free,  we get 
	\begin{align*}
	&\langle-(y_1\cdot \nabla)y_1+(y_2\cdot \nabla)y_2+\text{ div}(N(y_1)-N(y_2)), \mathbb{P}v(y)\rangle\\
	&= \langle-(y\cdot \nabla)v(y_1)-(y_2\cdot \nabla)v(y)-\sum_{j}v(y)^j\nabla y_1^j-\sum_{j}v(y_2)^j\nabla y^j+(\alpha_1+\alpha_2)\text{ div}(A(y_1)^2-A(y_2)^2), \mathbb{P}v(y)\rangle.
	\end{align*}
	Therefore, we infer that
	\begin{align*}
	&- 2\langle \big[(y\cdot \nabla)y_1+(y_2\cdot \nabla)y\big],\mathbb{P}v(y) \rangle
	+2\langle \text{ div}(N(y_1)-N(y_2)), \mathbb{P}v(y)\rangle\vspace{2mm}	\\
	&=-2b(y,v(y_1),\mathbb{P}v(y))-2b(y_2,v(y)-\mathbb{P}v(y),\mathbb{P}v(y))-2b(\mathbb{P}v(y),y_1,v(y))\\[1mm]&\quad-2b(\mathbb{P}v(y),y,v(y_2))+2(\alpha_1+\alpha_2)(\text{ div}[A(y_1)A(y)+A(y)A(y_2)],\mathbb{P}v(y))\vspace{2mm}\\
	&\leq C\Vert y \Vert_\infty\Vert v(y_1)\Vert_{H^{1}}\Vert \mathbb{P}v(y)\Vert_2+C\Vert y_2 \Vert_\infty\Vert v(y)-\mathbb{P}v(y)\Vert_{H^{1}}\Vert \mathbb{P}v(y)\Vert_2
	\vspace{2mm}\\ 
	&\quad+C\Vert y_1 \Vert_{W^{1,\infty}}\Vert v(y)\Vert_{2}\Vert \mathbb{P}v(y)\Vert_2 +C\Vert y\Vert_{W^{1,4}}\Vert v(y_2)\Vert_{4}\Vert \mathbb{P}v(y)\Vert_2\vspace{2mm}\\
	&\quad+C \Vert \text{ div}[A(y_1)A(y)+A(y)A(y_2)]\Vert_2 \Vert Pv(y) \Vert_{2}
	\leq C(\Vert y_1\Vert_{H^3}+\Vert y_2\Vert_{H^3})\Vert y\Vert_{W}^2,
	\end{align*}
	where we used that $H^3(D)\hookrightarrow W^{2,4}(D)\cap W^{1,\infty}(D)$.
	On the other hand, we have
	\begin{equation*}
	\begin{array}{ll}
	2\langle \text{ div}(S(y_1)-S(y_2)),\mathbb{P}v(y)\rangle&=2\beta\left(\text{ div}\left\{\vert A(y_1)\vert^2A(y)
	+\left[A(y_1): A(y)+A(y):A(y_2)\right]
	A(y_2) \right\},\mathbb{P}v(y)\right)\vspace{2mm}\\
	&\leq C(\Vert y_1\Vert_{H^3}^2+\Vert y_2\Vert_{H^3}^2)\Vert y\Vert_{W}^2,
	\end{array}
	\end{equation*}
	where we used a similair arguments to the one used to  get \eqref{estimateWzn} to estimate the last two terms.
	Hence, there exists $M>0$ such that 
	\begin{align}\label{stabilityWnorm}
	\Vert \mathbb{P} v (y(t))\Vert_2^2
	&\leq \Vert \mathbb{P} v (y_0)\Vert_2^2+C\int_0^t(1+\sum_{l=1}^2[\Vert y_l\Vert_{H^3}+\Vert y_l\Vert_{H^3}^2])\Vert y\Vert_{W}^2 ds+ \int_0^t\Vert U_1-U_2\Vert_{2}^2ds.
	\end{align}
	Consequently, \eqref{stabilityVnorm} and \eqref{stabilityWnorm} give the following relation
	\begin{align*}
	&\Vert \mathbb{P} v (y(t))\Vert_2^2+\Vert y(t)\Vert_V^2+4\nu \int_0^t\Vert \mathbb{D}y\Vert_{2}^2ds \vspace{2mm}\\
	&\quad\leq \Vert y_0\Vert_{V}^2+\Vert \mathbb{P} v (y_0)\Vert_2^2+ M\int_0^t(1+\Vert y_1\Vert_{H^3}^2+\Vert y_2\Vert_{H^3}^2)\Vert y\Vert_{W}^2 ds+ M\int_0^t\Vert U_1-U_2\Vert_{2}^2ds.
	\end{align*}
	As a conclusion, we obtain
	\begin{align*}
	&\sup_{t\in [0, T]}\Vert y(t)\Vert_W^2
\leq  \Vert y_0\Vert_W^2	
	+M\int_0^T(1+\Vert y_1\Vert_{H^3}^2+\Vert y_2\Vert_{H^3}^2)\Vert y\Vert_{W}^2 ds+ M\int_0^T\Vert U_1-U_2\Vert_{2}^2ds,
	\end{align*}
	where $K$ is a positive constant. Finally,  \eqref{stability-estimate} is a consequence of Gronwall's inequality with 
	$$C=(M+1)e^{M(1+2K)T}, \quad K=\displaystyle\sup_{r\in [0, T]}M_4^2(t). $$
\end{proof}
\section{Gâteaux differentiability of the control-to-state mapping}
\label{S6}
This section studies  the differentialility  of the control-to-state mapping. More precisely, with the help of the stability property established in the previous section, we will prove that
the Gâteaux derivative of the control-to-state mapping is provided by the solution of the linearized equation.

\begin{prop}\label{prop-gateux-diff}
	Let us consider $U$ and $y_0$ satisfying \eqref{H1} and $\psi\in L^2(0,T;(H^1(D))^2)$. Defining
	$$ U_\rho=U+\rho\psi, \quad \rho \in (0,1),$$
	let $y$ and $y_\rho$ be  the solutions of \eqref{I} associated with $(U,y_0)$ and $(U_\rho,y_0)$, respectively, then the following representaion holds
	\begin{align}\label{representationz}
	y_\rho=y+\rho z+\rho \delta_\rho \;\text{ with } \;\displaystyle\lim_{\rho\to 0} \sup_{t\in [0,T]}\Vert \delta_\rho\Vert_V^2=0,
	\end{align}
	where $z\in L^\infty(0,T;W)$  is the solution of  \eqref{Linearized}, satisfying the estimates of Proposition \ref{exis-linearize}.
\end{prop}

\begin{proof}
	We recall that $y$ verifies the equation
	\begin{equation}
	\begin{array}{ll}
	\partial_t(v(y))&=-\nabla \mathbf{\bar P}+\nu \Delta y-(y\cdot \nabla)v-\sum_{j}v^j\nabla y^j+(\alpha_1+\alpha_2)\text{div}(A^2)\vspace{2mm}\\
	& \qquad+\beta \text{div}(|A|^2A)+U. 
	\end{array}
	\end{equation}
	Therefore
	\begin{equation}
	\begin{array}{ll}
	\label{Gat-equ}
	\partial_t(v(y_\rho-y))=&-\nabla ( \mathbf{\bar P}_\rho-\mathbf{\bar P})+\nu \Delta (y_\rho-y)-((y_\rho\cdot \nabla)v_\rho-(y\cdot \nabla)v)\vspace{2mm}\\
	&-\sum_{j}(v_\rho^j\nabla y_\rho^j-v^j\nabla y^j)
	+(\alpha_1+\alpha_2)\text{div}(A_\rho^2-A^2)\vspace{2mm}\\
	&+\beta\text{ div} \left ( |A(y_\rho)|^2A(y_\rho)-|A(y)|^2A(y)\right)+\rho\psi,
	\end{array}
	\end{equation}
	where
	$	v_\rho=v(y_\rho),\quad v=v(y),\quad A_\rho=A(y_\rho),\quad A=A(y).$\\
	
	Setting $z_\rho=\frac{y_\rho-y}{\rho}$,
	$\pi_\rho=\frac{ \mathbf{\bar P}_\rho- \mathbf{\bar P}}{\rho}$,
	we  notice that  $z_\rho$ is the unique solution for the following equation
	$$
	\begin{array}{ll}
	\partial_t(v(z_\rho))=&\psi-\nabla \pi_\rho+\nu \Delta z_\rho-\left[(z_\rho\cdot\nabla) v(y_\rho)+(y\cdot\nabla) v(z_\rho)\right]\vspace{2mm}\\
	&-\sum_{j}[v^j(z_\rho)\nabla y_\rho^j+v^j(y)\nabla z_\rho^j]+\beta\text{div} \bigl( |A(y_\rho)|^2A(z_\rho)
	+[A(z_\rho):A(y_\rho)\vspace{2mm}\\
	&+A(y):A(z_\rho)]A(y)\bigr)\vspace{2mm}+(\alpha_1+\alpha_2)\text{div}
	\left[A(z_\rho)A(y_\rho)+A(y)A(z_\rho)\right].
	\end{array}
	$$
	Defining   $\delta_\rho=z_\rho-z$, the following equation  holds
	\begin{equation}
	\begin{array}{ll}
	\label{delta}
	\hspace*{-0.3cm}	\partial_t(v(\delta_\rho))&=-\nabla (\pi_\rho-\pi)+\nu \Delta \delta_\rho-
	\left[(y\cdot\nabla) v(\delta_\rho)+(\delta_\rho\cdot \nabla)v(y_\rho)\right]-(z\cdot\nabla) v(y_\rho-y)
	\vspace{2mm}\\
	&-\sum_{j}[v^j(y)\nabla\delta_\rho^j+v^j(\delta_\rho)\nabla y_\rho^j+v^j(z)\nabla (y_\rho-y)^j]\vspace{2mm}\\
	&+(\alpha_1+\alpha_2)\text{div}\left[A(y)A(\delta_\rho)+A(\delta_\rho)A(y_\rho)+A(z)A(y_\rho-y)\right]\vspace{2mm}\\
	&+\beta\text{ div} \left\{\left( A(y):A(\delta_\rho)\right)A(y)+ |A(y_\rho)|^2A(\delta_\rho)\right\}\vspace{2mm}\\
	&+\beta\text{ div} (\left[ A(y_\rho-y):A(y_\rho)+A(y): A(y_\rho-y)\right]A(z)])\vspace{2mm}\\
	&
	+\beta\text{ div} \left[ \left(A(\delta_\rho):A(y_\rho)+A(z): A(y_\rho-y)\right)A(y)\right]=:g(\delta_\rho).
	\end{array}
	\end{equation}
	Multiplying this  equation by $\delta_\rho$, we write
	$ \partial_t\Vert\delta_\rho\Vert_{V}^2=2(g(\delta_\rho),\delta_\rho).$ Let us estimate the right hand side. 
	\begin{equation*}
	\begin{array}{ll}
	2(g(\delta_\rho),\delta_\rho)
	&=-4\nu \Vert \mathbb{D}\delta_\rho\Vert_{2}^2-2b(y,v(\delta_\rho),\delta_\rho)-2b(\delta_\rho,v(y_\rho),\delta_\rho)-2b(z,v(y_\rho-y),\delta_\rho)\vspace{2mm}\\&\quad-2b(\delta_\rho,\delta_\rho,v(y))
	-2b(\delta_\rho,y_\rho,v(\delta_\rho))-2b(\delta_\rho,(y_\rho-y),v(z))\vspace{2mm}\\
	&\quad+2(\alpha_1+\alpha_2)\left(\text{div}[A(y)A(\delta_\rho)+A(\delta_\rho)A(y_\rho)+A(z)A(y_\rho-y)],\delta_\rho\right)\vspace{2mm}\\
	&\quad
	+2\beta\left(\text{ div} \left[(A(y): A(\delta_\rho))A(y)+ |A(y_\rho)|^2A(\delta_\rho)\right],\delta_\rho\right)\vspace{2mm}\\
	&\quad+2\beta\left(\text{ div} \left[\left( A(y_\rho-y):A(y_\rho)+A(y): A(y_\rho-y)\right)A(z)\right],\delta_\rho\right)\nonumber\vspace{2mm}\\
	&\quad
	+2\beta\big(\text{ div}\left[\left(A(\delta_\rho):A(y_\rho)+A(z): A(y_\rho-y)\right)A(y)\right],\delta_\rho\big)\nonumber\vspace{2mm}\\
	&\quad=-4\nu \Vert \mathbb{D}\delta_\rho\Vert_{2}^2+R_1+R_2+R_3.
	\end{array}
	\end{equation*}
	We have 
	\begin{equation*}
	\begin{array}{ll}
	R_1=&-2b(y,v(\delta_\rho),\delta_\rho)-2b(\delta_\rho,v(y_\rho),\delta_\rho)-2b(z,v(y_\rho-y),\delta_\rho)\vspace{2mm}\\
	&\quad-2b(\delta_\rho,\delta_\rho,v(y))
	-2b(\delta_\rho,y_\rho,v(\delta_\rho))-2b(\delta_\rho,(y_\rho-y),v(z))\vspace{2mm}\\
	&=2b(y,\delta_\rho,v(\delta_\rho))-2b(\delta_\rho,y_\rho,v(\delta_\rho))-[2b(\delta_\rho,v(y_\rho),\delta_\rho)+2b(z,v(y_\rho-y),\delta_\rho)]\vspace{2mm}\\&\quad-2b(\delta_\rho,\delta_\rho,v(y))-2b(\delta_\rho,(y_\rho-y),v(z))\vspace{2mm}\\
	&\leq C\Vert y\Vert_{H^2}\Vert\delta_\rho\Vert_V^2+C\Vert y_\rho\Vert_{H^3}\Vert\delta_\rho\Vert_V^2+C\Vert \delta_\rho\Vert_{4}^2\Vert v(y_\rho)\Vert_{H^1}+C\Vert z\Vert_{\infty}\Vert \delta_\rho\Vert_V\Vert v(y_\rho-y)\Vert_{2}\vspace{2mm}
	\\
	&\quad+C\Vert\delta_\rho\Vert_{4}\Vert \delta_\rho\Vert_V\Vert v(y)\Vert_{4}+C\Vert \delta_\rho\Vert_{4}\Vert y_\rho-y\Vert_{W^{1,4}}\Vert v(z)\Vert_{2},
	\end{array}
	\end{equation*}
	where, to estimate the first two terms,
	we adapted the arguments used in \cite[(3.22)-(3.23)]{CC18},   and for the other terms we applied the H\"older inequality. Consequently, we derive
	\begin{equation*}
	\begin{array}{ll}
	R_1\leq C(\Vert y\Vert_{H^3}+\Vert y_\rho\Vert_{H^3})\Vert \delta_\rho\Vert_V^2+C\Vert z\Vert_{H^2}^2\Vert \delta_\rho\Vert_V^2+C\Vert y_\rho-y\Vert_{H^2}^2.
	\end{array}
	\end{equation*}
	Using the Stokes theorem and the boundary conditions for  $\delta_\rho$, we deduce
	\begin{align*}
R_2	&=-2(\alpha_1+\alpha_2)\int_D[A(y)A(\delta_\rho)+A(\delta_\rho)A(y_\rho)+A(z)A(y_\rho-y)]:\nabla\delta_\rho dx\vspace{2mm}\\
	&\leq C\Vert y\Vert_{W^{1,\infty}}\Vert \delta_\rho\Vert_V^2+C\Vert y_\rho\Vert_{W^{1,\infty}}\Vert \delta_\rho\Vert_V^2+C\Vert z\Vert_{W^{1,4}}\Vert y_\rho-y\Vert_{W^{1,4}}\Vert\delta_\rho\Vert_V\vspace{2mm}\\
	&\leq C(\Vert y\Vert_{H^3}+\Vert y_\rho\Vert_{H^3})\Vert \delta_\rho\Vert_V^2+C\Vert z\Vert_{H^2}^2\Vert \delta_\rho\Vert_V^2+C\Vert y_\rho-y\Vert_{H^2}^2.
	\end{align*}
	Analogous  arguments give 
	\begin{align*}
	R_3
	=&-2\beta\int_D \left\{A(y):A(\delta_\rho)A(y)+ |A(y_\rho)|^2A(\delta_\rho)\right\}:\nabla\delta_\rho dx\vspace{2mm}\\
	& -2\beta\int_D\left\{\left[ A(y_\rho-y):A(y_\rho)+A(y): A(y_\rho-y)\right]A(z)\right\}:\nabla\delta_\rho dx\vspace{2mm}\\
	&\quad
	-2\beta\int_D \left[ \left(A(\delta_\rho): A(y_\rho)+A(z): A(y_\rho-y)\right)A(y)\right]:\nabla\delta_\rho dx\vspace{2mm}\\
	&\leq C\left(\Vert y\Vert_{H^3}^2+\Vert y_\rho\Vert_{H^3}^2\right)\left(1+\Vert z\Vert_{H^2}^2\right)\Vert \delta_\rho\Vert_V^2+C\Vert y_\rho-y\Vert_{H^2}^2,
	\end{align*}
	where we used H\"older and Young inequalities to deduce the last estimate.
	Summing up, we obtain
	\begin{equation*}
	\begin{array}{ll}
	\partial_t\Vert \delta_\rho\Vert_V^2&\leq 
	C\left(1+\Vert y\Vert_{H^3}^2+\Vert y_\rho\Vert_{H^3}^2\right)\left(1+\Vert z\Vert_{H^2}^2\right)\Vert \delta_\rho\Vert_V^2+C\Vert y_\rho-y\Vert_{H^2}^2\\[0.15cm]
	&\leq \widetilde K_0\Vert \delta_\rho\Vert_V^2+C\Vert y_\rho-y\Vert_{H^2}^2,
	\end{array}
	\end{equation*}
	by Lemma \ref{Lemma-H3} and Proposition \ref{exis-linearize}. Finally, 
	Gronwall's inequality yields
	\begin{equation*}
	\sup_{s\in [0, T]}\Vert\delta_\rho(s)\Vert_V^2
	\leq C\int_0^T\Vert y_\rho-y\Vert_{H^2}^2ds.
	\end{equation*}
	Applying  \eqref{stability-estimate} with $y_1=y$ and $y_2=y_\rho$, we derive	
	\begin{equation*}
	\sup_{s\in [0, T]}\Vert\delta_\rho(s)\Vert_V^2
	\leq  C\rho^2\int_0^T\Vert \psi\Vert_{2}^2ds.
	\end{equation*}
	Now,  taking  $\rho \to 0$ we infer  
	\eqref{representationz}, i.e.
	\begin{equation*}
	\begin{array}{ll}
	y_\rho=y+\rho z+\rho \delta_\rho \text{ with } \displaystyle\lim_{\rho\to 0} \sup_{t\in [0,T]}\Vert \delta_\rho\Vert_V^2=0.		
	\end{array}
	\end{equation*}
\end{proof}
\subsection{Variation of the cost functional \eqref{cost-uniqueness}}
As a consequence of Proposition \ref{prop-gateux-diff}, we get the following result on the variation for the cost functional \eqref{cost-uniqueness}.
\begin{prop}\label{vari-cost}
	Let us consider $U,\,y_0, \,\psi$ and $U_\rho=U+\rho\psi$ verifying 
	the hypothesis of the Proposition 
	\ref{prop-gateux-diff}. Then  
	\begin{equation*}
	J(U_\rho,y_\rho)=J(U,y)+\rho\int_0^T\{ (\nabla_uL(t,U(t),y(t)),\psi(t))+(\nabla_yL(t,U(t),y(t)),z(t))\}dt+o(\rho),
	\end{equation*}
	where $y_\rho,y$ are the solutions of \eqref{I}, corresponding to $(U_\rho,y_0)$ and $(U,y_0)$, respectively, and $z$ is the solution of \eqref{Linearized}.
\end{prop}
\section{Adjoint equation}
\label{S7}
Let $f\in (L^2(D\times ]0,T[))^2$. Our aim is to prove the well posedness of the adjoint equation given by
\begin{align}\label{adjoint}
\begin{cases}
-\partial_t(v(p))-\nu \Delta p
-\text{curl }v(y)\times p+\text{curl }v(y\times 
p)-(\alpha_1+\alpha_2)\text{div}\big[A(y)A(p)+A(p)A(y)\big]&\\[1mm]
\qquad\qquad-\beta\text{div}\big[\vert A(y)\vert^2A(p)\big]-2\beta\text{div}\big[(A(y): A(p))A(y)\big]=f-\nabla \pi &\hspace*{-2.75cm}\text{in } D\times (0,T),\\[1mm]
\text{div}(p)=0 \quad &\hspace*{-2.75cm}\text{in } D\times (0,T),\\
p\cdot \eta=0, \quad [\eta \cdot \mathbb{D}(p)]\cdot \tau=0  \quad &\hspace*{-2.75cm}\text{on } \partial D\times (0,T),\\
p(T)=0 \quad &\hspace*{-2.75cm}\text{in } D.
\end{cases}
\end{align}
\begin{definition}\label{Def-adjoint}
	A function $p \in L^{\infty}(0,T;W)$ with $\partial_tp \in L^{2}(0,T;V)$ is a solution of \eqref{adjoint} if $p(T)=0$ and  for any $t\in [0,T]$, the following equality holds
	\begin{equation*}
	\begin{array}{ll}
	&(-\partial_tv(p),\phi)+2\nu(\mathbb{D} p,
	\mathbb{D}\phi)-b(\phi,p,v(y))+b(p,\phi,v(y))+b(p,y,v(\phi))\vspace{2mm}\\&
	\quad \quad -b(y,p,v(\phi))+(\alpha_1+\alpha_2)\big(A(y)A(p)+A(p)A(y),\nabla \phi\big)  +\beta\big( |A(y)|^2A(p), \nabla\phi\big)\vspace{2mm}\\&\qquad+2\beta \big(\left(A(p):	A(y)\right)A(y),\nabla\phi\big)= (f,\phi), \quad 
	\quad \text{for all } \phi \in W.
	\end{array}
	\end{equation*}
\end{definition}
Let us state  the following result about the solution of \eqref{adjoint}.
\begin{prop}\label{exis-adjoint}
	Let	$f\in (L^2(D\times ]0,T[))^2$, then there exists a unique solution $p$ to \eqref{adjoint} in the sense of Definition \ref{Def-adjoint}
	satisfying the following estimates
	\begin{align*}
	\sup_{r\in [0,t]}\Vert p(r)\Vert_V^2+4\nu \int_0^{ t}\Vert  p(s)\Vert_V^2ds&\leq C(T)\int_0^{T}\Vert f(s)\Vert_{2}^2ds\quad \text{  for  any } t\in ]0,T];\vspace{2mm}\\
	\sup_{r\in [0, t]}\Vert p(r)\Vert_W^2 &\leq C(T)\int_0^T\Vert f(s)\Vert_2^2ds;\vspace{2mm}\\
	\int_0^T(\Vert \partial_tp\Vert_2^2+\alpha_1\Vert  \mathbb{D}\partial_tp\Vert_2^2)ds&\leq  C_1(T).	
	\end{align*}
\end{prop}
\subsection{Proof of Proposition \ref{exis-adjoint}}
Notice  that $p$ is the solution of \eqref{adjoint} if and only if $q(t)=p(T-t)$ is the solution of the following initial value problem with $\bar y(t)=\bar y(T-t), \bar f(t)= f(T-t)$ and $ \bar \pi(t)=\pi(T-t)$
\begin{equation}
\label{adjoint-initial}
\left\{
\begin{array}{ll}
\partial_t(v(q))-\nu \Delta q
-\text{curl }v(\bar y)\times q+\text{curl }v( \bar y\times 
q)-(\alpha_1+\alpha_2)\text{div}\big[A(\bar y)A(q)+A(q)A(\bar y)\big]&\vspace{2mm}\\
\qquad\qquad-\beta\text{div}\big[\vert A(\bar y)\vert^2A(q)\big]-2\beta\text{div}\big[(A(\bar y):A(q))A(\bar y)\big]=\bar f-\nabla \bar\pi &\hspace*{-2.75cm}\text{in } D\times (0,T),\vspace{2mm}\\
\text{div}(q)=0 \quad &\hspace*{-2.75cm}\text{in } D\times (0,T),\vspace{2mm}\\
q\cdot \eta=0, \quad [\eta \cdot \mathbb{D}(q)]\cdot \tau=0  \quad &\hspace*{-2.75cm}\text{on } \partial D\times  (0,T),\vspace{2mm}\\
q(0)=0 \quad &\hspace*{-2.75cm}\text{in } D.
\end{array}
\right.
\end{equation}
According to the  Definition \ref{adjoint}, $q$ is the solution of \eqref{adjoint-initial} if 
$q \in L^{\infty}(0,T;W)$ with $\partial_tq \in L^{2}(0,T;V)$, $q(0)=0$ and, in addition, the following equality holds, for all  $\phi \in W$,
\begin{equation*}
\begin{array}{ll}
&(\partial_tv(q),\phi)+2\nu(\mathbb{D} q,
\mathbb{D}\phi)-b(\phi,q,v(\bar y))+b(q,\phi,v(\bar y))+b(q,\bar y,v(\phi))-b(\bar y,q,v(\phi))\vspace{2mm}\\
&+(\alpha_1+\alpha_2)\big(A(\bar y)A(q)+A(q)A(\bar y),\nabla \phi\big)  +\beta\big( |A(\bar y)|^2A(q), \nabla\phi\big)\vspace{2mm}\\
&\quad+2\beta \big((A(q):A(\bar y))A(\bar y),\nabla\phi\big)= (\bar f,\phi).
\end{array}
\end{equation*}
To study the equation \eqref{adjoint-initial}, we will  follow closely the  analysis used in  Section \ref{Linearized-section} to study the linearized equation. For that, consider $W_n=\text{span}\{h_1,\cdots,h_n\}$ and define the corresponding approximations
$  q_n(t)=\sum_{i=1}^n d_i(t)h_i$   for each  $t\in [0,T].$
The approximated problem for \eqref{adjoint-initial}  can be written as $q_n(0)=0$ and 
\begin{align}\label{approximation-adjoint}
&(\partial_tv(q_n),\phi)+2\nu(\mathbb{D} q_n,\mathbb{D}\phi)-b(\phi,q_n,v(\bar y))+b(q_n,\phi,v(\bar y))+b(q_n,\bar y,v(\phi))\\&\hspace*{0.25cm}-b(\bar y,q_n,v(\phi))
+(\alpha_1+\alpha_2)\big(A(\bar y)A(q_n)+A(q_n)A(\bar y),\nabla \phi\big)  +\beta\big( |A(\bar y)|^2A(q_n), \nabla\phi\big)\nonumber\\
&\qquad+2\beta \big((A(q_n):
A(\bar y))A(\bar y),\nabla\phi\big)= (\bar f,\phi), \quad \text{ for any }\phi \in W_n. \nonumber
\end{align}
Now, we remark that the structure of \eqref{approximation-adjoint} and \eqref{approximationz-n} are similar. Therefore, by adapting the   arguments used to derive  Lemma \ref{Lem-esti-zn}, we are able to deduce the  following result for  \eqref{approximation-adjoint}.
\begin{lemma}\label{Lem-esti-qn}
	Let $f\in (L^2(D\times ]0,T[))^2$. Then there exists a unique  solution $q_n\in \mathcal{C}([0,T],W_n)$ to \eqref{approximation-adjoint}. Moreover, there exist  positive constants $C(T),C_1(T)$, which are independent on the index $n$, such that the following estimates hold for each $t\in [0,T]$
	\begin{align*}
	\sup_{r\in [0,t]}\Vert q_n(r)\Vert_V^2+4\nu \int_0^{ t}\Vert \mathbb{D} q_n(s)\Vert_2^2ds&\leq C(T)\int_0^{T}\Vert f(s)\Vert_{2}^2ds;\vspace{2mm}\\
	\sup_{r\in [0, t]}\Vert q_n(r)\Vert_W^2 &\leq C(T)\int_0^T\Vert f(s)\Vert_2^2ds;\vspace{2mm}\\
	\int_0^T(\Vert \partial_tq_n\Vert_2^2+\alpha_1\Vert  \mathbb{D}\partial_tq_n\Vert_2^2)ds&\leq  C_1(T).
	\end{align*}
\end{lemma}
Consequentlty, Proposition \ref{exis-adjoint} follows by passing to the limit in \eqref{approximation-adjoint}.

\section{Existence of optimal control and optimality condition}
\label{S8}
This section starts with the presentation of  a  duality relation
between the solution of the linearized equation and the solution of the adjoint equation and  next shows that the control problem has a solution.  Taking into account the duality relation, it will be proved that the  solution of the control problem satisfies 
the first order optimality condition. 
\subsection{Duality property}
\begin{prop}\label{duality-prop}
	Let $y\in L^\infty(0,T;\widetilde{W})$ and $f,\psi \in (L^2(D\times ]0,T[))^2$. Then we have
	$$\int_0^T(\psi(t),p(t))dt=\int_0^T(f(t),z(t))dt, $$
	where $p$ is the solution of \eqref{adjoint} and $z$ is the solution of \eqref{Linearized}.
\end{prop}
\begin{proof} For  any  $t\in [0,T]$, let $p_n(t)=q_n(T-t)$, where $q_n$ is the solution of \eqref{approximation-adjoint}. Then 
	$  p_n(t)=\sum_{i=1}^n \bar d_i(t)h_i$
	verifies  $p_n(T)=0$ and
	\begin{align}\label{pndual}
	\begin{cases}
	&\hspace*{-0.22cm}(-\partial_tv(p_n),\phi)+2\nu(\mathbb{D} p_n,\mathbb{D}\phi)-b(\phi,p_n,v( y))+b(p_n,\phi,v( y))+b(p_n, y,v(\phi))\\[1mm]&-b(y,p_n,v(\phi))
	+(\alpha_1+\alpha_2)\big(A( y)A(p_n)
	+A(p_n)A( y),\nabla \phi\big) \\[1mm] &+\beta\big( |A( y)|^2A(p_n), \nabla\phi\big)
	+2\beta \big((A(p_n):
	A(y))A( y),\nabla\phi\big)= (f,\phi),  \text{ for any }\; \phi \in W_n.
	\end{cases}
	\end{align}
	On the other hand, let us recall that the solution of the linearized equation $z_n$ given by
	$  z_n(t)=\sum_{i=1}^n c_i(t)h_i$  
	satisfies $z_n(0)=0$ and 
	\begin{align*}
	\begin{cases}
	&(\partial_tv(z_n),\phi)= \Big(\psi+\nu \Delta z_n-(y\cdot \nabla )v(z_n)-(z_n\cdot \nabla)v(y)-\sum_{j}v(z_n)^j\nabla y^j\\[1mm]&\qquad-\sum_{j}v(y)^j\nabla z_n^j
	+(\alpha_1+\alpha_2)\text{div}[A(y)A(z_n)+A(z_n)A(y)] \\[1mm]&\qquad+\beta\text{ div}\big [ |A(y)|^2A(z_n)\big]+2\beta\text{ div} 
	\big [(A(z_n):A(y))A(y)\big],\phi\Big)
	,   \text{  for  any }\,\phi \in W_n. 
	\end{cases}
	\end{align*}
	This means that $z_n(0)=0$ and
	\begin{align}\label{zndual}
	&(\partial_tv(z_n),\phi)+2\nu(\mathbb{D}
	z_n,\mathbb{D}\phi)+b(y,v(z_n),\phi)+b(z_n,v(y),\phi)+b(\phi,y,v(z_n))\\
	&\qquad+b(\phi,z_n,v(y))
	+(\alpha_1+\alpha_2)\big(A(y)A(z_n)+A(z_n)A(y),\nabla \phi\big)\nonumber\\[1mm]
	&\qquad	 +\beta\big( |A(y)|^2A(z_n), \nabla\phi\big)
	+2\beta \big((A(z_n):
	A(y))A(y),\nabla\phi\big)= (\psi,\phi),
	\text{ for all } \,\phi \in W_n.\nonumber
	\end{align}
	Setting $\phi=h_i$ in \eqref{pndual},   multiplying \eqref{pndual} by $c_i(t)$ and summing from $i=1$ to $n$, we get
	\begin{align}\label{p_nfinal}
	&(-\partial_tv(p_n),z_n)+2\nu(\mathbb{D} p_n,\mathbb{D}z_n)-b(z_n,p_n,v( y))+b(p_n,z_n,v( y))\\[1mm]&+b(p_n, y,v(z_n))-b(y,p_n,v(z_n))+(\alpha_1+\alpha_2)\big(A( y)A(p_n)+A(p_n)A( y),\nabla z_n\big) \nonumber\\[1mm]& +\beta\big( |A( y)|^2A(p_n), \nabla z_n\big)+2\beta \big((A(p_n):
	A(y))A( y),\nabla z_n\big)= ( f,z_n)\nonumber. 
	\end{align}
	Similarly, taking $\phi=h_i$ in \eqref{zndual}, multiplying \eqref{zndual} by $\bar{d}_i(t)$ and summing from $i=1$ to $n$, we obtain
	\begin{align}\label{z_nfinal}
	&(\partial_tv(z_n),p_n)+2\nu(\mathbb{D} z_n,\mathbb{D}p_n)+b(y,v(z_n),p_n)+b(z_n,v(y),p_n)\\&+b(p_n,y,v(z_n))
	+b(p_n,z_n,v(y))+(\alpha_1+\alpha_2)\big(A(y)A(z_n)+A(z_n)A(y),\nabla p_n\big)\nonumber \\ &+\beta\big( |A(y)|^2A(z_n), \nabla p_n\big)+2\beta \big((A(z_n):A(y))A(y),\nabla p_n\big)= (\psi,p_n)\nonumber.
	\end{align}
	A standard integration by part ensures that
	\begin{align*}
	(-\partial_tv(p_n(t)),z_n(t))= (\partial_tv(z_n(t)),p_n(t))-\dfrac{d}{dt}\big((p_n(t),z_n(t))+2\alpha_1(\mathbb{D}p_n(t),
	\mathbb{D}z_n(t)) \big).
	\end{align*}
	Integrating with respect to the time variable on the interval
	$[0,T]$ and using that $z_n(0)=0, \quad p_n(T)=0$, we derive
	$$
	\int_0^T(-\partial_tv(p_n(t)),z_n(t))dt= \int_0^T(\partial_tv(z_n(t)),p_n(t))dt.$$
	Now, combining \eqref{p_nfinal}  and \eqref{z_nfinal}, 
	and  using the fact that $(A,B)=(A^T,B^T)$, for any $A,B \in \mathcal{M}_{2\times 2}(\mathbb{R})$, we  deduce
	$$\int_0^T(\psi(t),p_n(t))dt=\int_0^T(f(t),z_n(t))dt.$$
	Therefore, taking the limit as $n\to \infty$, the result of the Proposition \ref{duality-prop} holds. 
\end{proof}
Considering  $f=\nabla_yL(\cdot,U,y) \in (L^2(]0,T[\times D))^2$ in Proposition \ref{duality-prop}, we obtain 
\begin{cor} Under the assumtions of Proposition \ref{duality-prop}, the following duality relation  holds
	$$\int_0^T(\psi(t),p(t))dt=\int_0^T(\nabla_yL(t,U(t),y(t)) ,z(t))dt. $$
\end{cor}
\subsection{Existence of an optimal control for \eqref{Problem-uniq}}
Let $(U_n,y_n)_n$ be a minimizing sequence, notice that $(U_n)_n$ is uniformly bounded in the closed convex set $ \mathcal{U}_{ad} \subset L^2(0,T;(H^1(D))^2)$. On the other hand, denoting by $y_n$  the solution of \eqref{I}, where $U$ is replaced  by $U_n$, Theorem \ref{THm1} and Lemma \ref{Lemma-H3} ensures that $(y_n)_n$ is uniformly bounded in $L^\infty(0,T;\widetilde{W}) \cap H^1(0,T;V)$.\\

By compactness with respect to the weak and weak-$*$ topologies in the spaces involved in the following product space 
$$L^2(0,T;(H^1(D))^2) \times  \big(L^\infty(0,T;\widetilde{W}) \cap H^1(0,T;V) \big ), $$ there exists $(U,y) \in L^2(0,T;(H^1(D))^2) \times  \big(L^\infty(0,T;\widetilde{W}) \cap H^1(0,T;V) \big )$ such that the following convergences hold, up to sub-sequences (denoted by the sequences)
\begin{equation}
\label{weak-cv-yn}
\begin{array}{ll}
U_n &\rightharpoonup \tilde U \quad \text{  in }  L^2(0,T;(H^1(D))^2),\vspace{1mm}\\
y_n &\overset{\ast}{\rightharpoonup} \tilde y\quad \text{  in }  L^\infty(0,T;V),\vspace{1mm}\\
y_n &\overset{\ast}{\rightharpoonup} \tilde y \quad \text{  in }  L^\infty(0,T;W),\vspace{1mm}\\
y_n &\rightharpoonup \tilde y \quad \text{  in }  L^2(0,T;\widetilde{W}),\vspace{1mm}\\
\partial_ty_n &\rightharpoonup \partial_t \tilde y \quad \text{  in }  L^2(0,T;V).
\end{array}
\end{equation}
From \eqref{weak-cv-yn}, we  deduce  that $\tilde y\in \mathcal{C}([0,T],V)$ and therefore $y_n(0)=y(0)$ converges to $\tilde y(0)$ in $V$, which gives  $\tilde y(0)=y_0$.
Now, standard arguments (similar reasoning as in \cite{Bus-Ift-2}) ensure that $(\tilde U, \tilde y)$ solves \eqref{I}.\\

Recall that  $J: L^2(0,T;(H^1(D))^2)\times L^2(0,T;\widetilde{W}) \to \mathbb{R}^+$ given by \eqref{cost-uniqueness} is convex and  continuous. From \eqref{weak-cv-yn} we have
\begin{align*}
	U_n \rightharpoonup \tilde U \quad \text{  in }  L^2(0,T;(H^1(D))^2) \text{ and }
	y_n \rightharpoonup \tilde y\quad \text{  in }  L^2(0,T;\widetilde{W}).
\end{align*}
 The (weak)
lower semicontinuity of $J$ ensures
$$ J(\tilde U,\tilde y) \leq \liminf_n J(U_n,y_n), $$
which gives that $(\tilde U,\tilde y)$ is an optimal pair.
\subsection{A necessary optimality condition for \eqref{Problem-uniq}}
Let $(\tilde U, \tilde y)$ be the optimal control pair. Consider $\psi \in \mathcal{U}_{ad}$ and  define $U_\rho=\tilde{U}+\rho(\psi-\tilde{U})$.  Thanks to Proposition \ref{prop-gateux-diff} and Proposition \ref{vari-cost}, we have   
\begin{align*}
\dfrac{J(U_\rho,y_\rho)-J(\tilde U, \tilde y)}{\rho}=\int_0^T\{ (\nabla_uL(\cdot,\tilde U, \tilde y),\psi-\tilde{U})+(\nabla_yL(\cdot,\tilde U, \tilde y),z)\}dt+\dfrac{o(\rho)}{\rho}.
\end{align*}
Then, the  G\^ateaux derivative of the cost functional $J$  is given by
\begin{align*}
\lim_{\rho\to 0}\dfrac{J(U_\rho,y_\rho)-J(\tilde U, \tilde y)}{\rho}=\int_0^T\{ (\nabla_uL(\cdot,\tilde U, \tilde y),\psi-\tilde{U})+(\nabla_yL(\cdot,\tilde U, \tilde y),z)\}dt \geq 0.
\end{align*}
Therefore, we have
$$\int_0^T\{ (\nabla_uL(\cdot,\tilde U, \tilde y),\psi-\tilde{U})+(\nabla_yL(\cdot,\tilde U, \tilde y),z)\}dt \geq 0,$$
where $z$ is the unique solution to the linearized problem \eqref{Linearized} with $\psi$ replaced by $\psi-\tilde{U}$.
\vspace{2mm}\\
Let $\tilde p$ be  the unique solution of \eqref{adjoint}. The application of  Proposition \ref{duality-prop} yields
$$\int_0^T(\psi(t)-\tilde{U}(t),\tilde p(t))dt=\int_0^T(\nabla_yL(t,\tilde{U}(t),\tilde{y}(t)) ,z(t))dt. $$
Finally, we obtain the following optimality condition,  for any   $\psi \in \mathcal{U}_{ad}$
\begin{align}
\label{equation-duality}
&\int_0^T(\psi(t)-\tilde{U}(t),\tilde p(t)+\nabla_uL(t,\tilde{U}(t),\tilde{y}(t)))dt\nonumber\vspace{2mm}\nonumber\\&\quad=\int_0^T(\psi(t)-\tilde{U}(t),\tilde p(t))dt+\int_0^T (\nabla_uL(t,\tilde{U}(t),\tilde{y}(t)),\psi(t)-\tilde{U}(t))dt \geq 0.
\end{align}
The proof of Theorem \ref{main-thm} results from the  combination of the previous sections.\\
\section{Uniqueness of the optimal solution}\label{S-uniq}
In the previous section, 
we derived the coupled system  constituted by  the so-called  first order necessary optimality conditions for the control problem.
This means that, at this stage,  a solution of the coupled system is just a candidate for an optimal solution. Thus,   
the uniqueness of the solution of the coupled system is  a very important issue in determining the optimal solution. This section addresses this uniqueness problem for the 
cost functional  given by  the  quadratic Lagrangian \eqref{cost-uniqueness}, and is devoted to the proof 
of Theorem \ref{Thm-uniq}.
\vspace{2mm}\\
Let us consider the problem \eqref{Problem-uniq}  and $f=\nabla_yL(\cdot,U,y)$ in \eqref{adjoint}. Then
\begin{cor}
	There exists $\widetilde\lambda>0$ such that
	\begin{align}\label{estimate-adjoint}
	\sup_{r\in [0, T]}\Vert p(r)\Vert_W^2 \leq  C(T)\int_0^T\Vert y-y_d\Vert_2^2ds\leq C(T)\int_0^T( M_0^2(s)+\Vert y_d\Vert_2^2)ds:= \widetilde\lambda^2.
	\end{align}
\end{cor}
On the other hand, a standard computation leads to
\begin{align}
\exists \Gamma >0:\quad	\vert (\text{curl} v(z)\times z,\phi)\vert\leq \Gamma \Vert \phi\Vert_{H^2}\Vert z\Vert_{H^2}^2, \quad \forall \phi,z \in W.
\end{align}
\subsection{Proof of Theorem \ref{Thm-uniq}}
Let $U_1,U_2$ be two optimal control variables for \eqref{Problem-uniq} and $y_1,y_2$ be the corresponding optimal states with the adjoint states $p_1,p_2$.\\

Now, let us consider  $y=y_1-y_2, \mathbf{P}=\mathbf{P}_1-\mathbf{P}_2$, $U=U_1-U_2$  and notice that $y$ solves the equation
\begin{equation}
\label{dif1}
\begin{array}{ll}
\partial_t(v(y))-\nu \Delta y+(y\cdot \nabla)v(y)+(y\cdot\nabla) v(y_2)+(y_2\cdot\nabla )v(y)&\\[2mm]\qquad\qquad+\sum_{j}\big[v(y)^j\nabla y^j+v(y)^j\nabla y_2^j+v(y_2)^j\nabla y^j\big]&\\[2.5mm]
\qquad\qquad -(\alpha_1+\alpha_2)\text{div}\left[(A(y))^2+A(y)A(y_2)+A(y_2)A(y)\right]
&\\[2.5mm]
\qquad\qquad -\beta \text{div}\left[|A(y)|^2A(y)+|A(y)|^2A(y_2)
+\left(A(y):A(y_2)\right)A(y)\right]&\\[2.5mm]
\qquad\qquad-\beta \text{div}\left[\left(A(y):A(y_2)\right)A(y_2)+\left(A(y_2):A(y)\right) A(y)\right]
&\\[2.5mm]
\qquad\qquad-\beta \text{div}\left[\left(A(y_2):A(y)\right) A(y_2)\right]-\beta \text{div}\left[\vert A(y_2)\vert^2A(y)\right]= -\nabla \mathbf{P} +U.&\\
\end{array}
\end{equation}
Let us multiply \eqref{dif1} by $p_2$ and  integrate  on $D$.
Then integrating by parts and taking into account the boundary conditions for the state and adjoint variables, we get the following variational formulation
\begin{equation}
\label{y-p2}
\begin{array}{ll}
(\partial_t y,p_2)_V=-2\nu(\mathbb{D}y,\mathbb{D} p_2)-b(y,v(y),p_2)-b(y,v(y_2),p_2)-b(y_2,v(y),p_2)-b(p_2,y,v(y))&\\
-b(p_2,y_2,v(y))-b(p_2,y,v(y_2))-\dfrac{1}{2}(\alpha_1+\alpha_2)((A(y))^2+A(y)A(y_2)+A(y_2)A(y),A(p_2))&\\
+(U,p_2)	-\dfrac{1}{2}\beta(|A(y)|^2A(y)+|A(y)|^2A(y_2)+\left(A(y): A(y_2)\right)A(y),A(p_2))&\\
-\dfrac{1}{2}\beta(\left(A(y):A(y_2)\right)A(y_2)
+\left(A(y_2):A(y)\right) A(y)+\left(A(y_2):A(y)\right) A(y_2)+\vert A(y_2)\vert^2A(y),A(p_2)).&
\end{array}
\end{equation}
Considering the adjoint equation for $p_2$, from Definition \ref{Def-adjoint}, and  by setting the test function $\phi=y$, we write
\begin{equation}
\label{p2-y}
\begin{array}{ll}
(\partial_tp_2,y)_V=&2\nu(\mathbb{D} p_2,
\mathbb{D}y)-b(y,p_2,v(y_2))+b(p_2,y,v(y_2))+b(p_2,y_2,v(y))-b(y_2,p_2,v(y))\vspace{2mm} \\
&-(y_2-y_d,y)+(\alpha_1+\alpha_2)\big(A(y_2)A(p_2)+A(p_2)A(y_2),\nabla y\big)\vspace{2mm} \\
& +\beta\big( |A(y_2)|^2A(p_2), \nabla y\big)+2\beta \big(\left(A(p_2): 	A(y_2)\right)A(y_2),\nabla y\big).
\end{array}
\end{equation}
Therefore summing the last two equalities \eqref{y-p2} and \eqref{p2-y}, we deduce
\begin{align*}
\partial_t(y,p_2)_V=&-(\text{curl} v(y)\times y,p_2)-(y_2-y_d,y)+(U,p_2)\\
&-(\alpha_1+\alpha_2)((A(y))^2,\nabla p_2)-\beta(\vert A(y)\vert^2A(y)+\vert A(y)\vert^2A(y_2),\nabla p_2)\\
&-2\beta(\left(A(y):A(y_2)\right)A(y),\nabla p_2),
\end{align*}
where we used the symmetry of $A$ and the bilinear form $b$ to get the last equality. By integrating from $t=0$ to $t=T$ and taking into account the initial and terminal conditions for $y$ and $p_2$, we have
\begin{align}
\label{uniq.1}
&0=-\int_0^T(\text{curl} v(y)\times y,p_2)-(y_2-y_d,y)+(U,p_2) dt-\int_0^T(\alpha_1+\alpha_2)((A(y))^2,\nabla p_2)dt\vspace{2mm}\nonumber\\&\qquad-\int_0^T\beta(\vert A(y)\vert^2A(y)+\vert A(y)\vert^2A(y_2),\nabla p_2)-2\beta(\left(A(y):A(y_2)\right)A(y),\nabla p_2)dt.
\end{align}
Analogously,   we can show 
that $\bar y=-y$
verifies the relation
\begin{align*}
\partial_t(\bar y,p_1)_V=&-(\text{curl} v(\bar y)\times \bar y,p_1))-(y_1-y_d,\bar y)-(U,p_1)\\[2mm]
&-(\alpha_1+\alpha_2)((A(\bar y))^2,\nabla p_1)-\beta(\vert A(\bar y)\vert^2A(\bar y)+\vert A(\bar y)\vert^2A(y_1),
\nabla p_1)\\[2mm]
&-2\beta(\left(A(\bar y):A(y_1)\right)A(\bar y),\nabla p_1).
\end{align*}
Integrating from $t=0$ to $t=T$, taking into account that $\bar y=-y$ and using the initial and terminal conditions for $y$ and $p_2$, we have
\begin{align}
\label{uniq.2}
0&=-\int_0^T(\text{curl} v(y)\times y,p_1))+(y_1-y_d,y)-(U,p_1) dt-\int_0^T(\alpha_1+\alpha_2)((A(y))^2,\nabla p_1)dt\vspace{2mm}\nonumber\\&\qquad+\int_0^T\beta(\vert A(y)\vert^2A(y)-\vert A(y)\vert^2A(y_1),\nabla p_1)-2\beta(\left(A(y):A(y_1)\right)A(y),\nabla p_1)dt.
\end{align}
By summing \eqref{uniq.1} and \eqref{uniq.2}, we  infer that
\begin{align*}
&\int_0^T\Vert y\Vert_{2}^2dt-\int_0^T(U,p_1-p_2)dt
\vspace{2mm}\\
&\quad= \int_0^T(\text{curl} v(y)\times y,p_1+p_2)dt+\int_0^T(\alpha_1+\alpha_2)((A(y))^2,\nabla p_1+\nabla p_2)dt\vspace{2mm}\\
&\quad-\int_0^T\beta(\vert A(y)\vert^2A(y)-\vert A(y)\vert^2A(y_1),\nabla p_1)+2\beta(A(y): A(y_1)A(y),\nabla p_1)dt\vspace{2mm}\\
&\quad+\int_0^T\beta(\vert A(y)\vert^2A(y)
+\vert A(y)\vert^2A(y_2),\nabla p_2)+2\beta(\left(A(y):A(y_2)\right)A(y),\nabla p_2)dt\vspace{2mm}\\
&\quad:=I_1+I_2+I_3.
\end{align*}
From \eqref{equation-duality}, the following optimality conditions  hold
\begin{align*}
\int_0^T(\psi-U_1, p_1+\lambda U_1)dt \geq 0,\quad
\int_0^T(\psi-U_2, p_2+\lambda U_2)dt \geq 0, \quad \forall \psi \in \mathcal{U}_{ad}.
\end{align*}
Setting $\psi=U_2$ and $\psi=U_1$  in the first and the second optimality conditions, respectively, we deduce
\begin{align*}
\lambda\int_0^T\Vert U\Vert_{2}^2dt \leq -\int_0^T(U,p_1-p_2)dt.
\end{align*}
On the other hand, \eqref{stability-estimate} and \eqref{estimate-adjoint} yield
\begin{align*}
\vert I_1\vert &\leq  \Gamma \int_0^T(\Vert p_1\Vert_{H^2}+\Vert p_2\Vert_{H^2})\Vert y\Vert_{H^2}^2 \leq 2\Gamma \widetilde{C}\widetilde{\lambda} \int_0^T\Vert U\Vert_2^2dt,\\
\vert I_2\vert &\leq 4(\alpha_1+\alpha_2)\int_0^T\Vert y\Vert_{W^{1,4}}^2(\Vert p_1\Vert_{H^2}+\Vert p_2\Vert_{H^2})dt\leq 8\kappa\widetilde{C}\widetilde{\lambda}(\alpha_1+\alpha_2) \int_0^T\Vert U\Vert_2^2dt,
\end{align*}
where $\kappa$ is a positive constant defined by the embedding $  H^2(D) \hookrightarrow W^{1,4}(D)$, 
\begin{align}\label{embedding-constant-uniq}
\Vert u\Vert_{W^{1,4}}^2\leq \kappa \Vert u\Vert_{W}^2, \quad \forall u \in W.
\end{align}
Thanks to  \eqref{state-H3-estimate}, \eqref{stability-estimate} and \eqref{estimate-adjoint}, by standard computations we obtain
\begin{align*}
\vert I_3\vert \leq 24\kappa\beta \widetilde{C}\widetilde{\lambda}\sup_{r\in [0, T]}M_4(r)\int_0^T\Vert U\Vert_2^2dt= 24\kappa\beta\gamma \widetilde{C}\widetilde{\lambda}\int_0^T\Vert U\Vert_2^2dt.
\end{align*}
Consequently, we have
\begin{align*}
\int_0^T\Vert y(t)\Vert_{2}^2dt+\lambda\int_0^T\Vert U(t)\Vert_{2}^2dt \leq  2 \widetilde{C}\widetilde{\lambda}[\Gamma+4\kappa(\alpha_1+\alpha_2)+12\kappa\beta\gamma]\int_0^T\Vert U(t)\Vert_2^2dt,
\end{align*}
which gives the claimed result.
\section*{Acknowledgment}
This work is funded by national funds through the FCT - Funda\c c\~ao para a Ci\^encia e a Tecnologia, I.P., under the scope of the projects UIDB/00297/2020 and UIDP/00297/2020 (Center for Mathematics and Applications).



\end{document}